\documentclass[final,leqno,letterpaper]{etna}

\usepackage{amsmath}
\usepackage{amssymb}
\usepackage{amstext}
\usepackage{graphicx}
\usepackage{appendix}
\usepackage{nomencl}
\usepackage{algorithm}
\usepackage{algpseudocode}
\usepackage{xifthen}
\usepackage{enumerate}
\usepackage[space,noadjust,nocompress]{cite}

\DeclareMathOperator{\sgn}{sgn}
\newcommand{\B}{\mathcal B}
\newcommand{\Kh}{\kappa}
\newcommand{\1}{\mathbf{1}}

\newcommand{\colind}[2]{\displaystyle\smash{\mathop{#1}^{\raisebox{.5\normalbaselineskip}{$\scriptstyle #2$}}}}

\setbibdata{1}{xx}{46}{2017} 

\hypersetup{%
    pdftitle={Heat conduction with phase change in permafrost modules for models},
    pdfauthor={David H\"otten},
    pdfkeywords={Stefan problem, enthalpy formulation, permafrost, decoupled energy conservation parametrization, piecewise affine function, Katzenelson algorithm}
    }

\title{Heat conduction with phase change in permafrost modules of vegetation models\thanks{%
Received... Accepted... Published online on... Recommended by....
}}

\author{David H\"otten\footnotemark[2]
        \and Jenny Niebsch\footnotemark[3]
        \and Ronny Ramlau \footnotemark[3] \and Walter Zulehner \footnotemark[3]}

\shorttitle{Heat conduction permafrost} 
\shortauthor{D.~H\"otten AND J.~Niebsch}

\begin{document}

\maketitle

\renewcommand{\thefootnote}{\fnsymbol{footnote}}

\footnotetext[2]{davidho@pik-potsdam.de, Potsdam Institute for Climate Impact Research, Telegrafenberg 31, 14473 Potsdam, Germany, corresponding author}
\footnotetext[3]{Radon Institute for Computational and Applied Mathematics, Altenbergerstraße 69, 4040 Linz, Austria}

\begin{abstract}
We consider the problem of heat conduction with phase change, that is essential for permafrost modeling in Land Surface Models and Dynamic Global Vegetation Models. These models require minimal computational effort and an extremely robust solver for large-scale, long-term simulations. The weak enthalpy formulation of the Stefan problem is used as the mathematical model and a finite element method is employed for the discretization. Leveraging the piecewise affine structure of the nonlinear time-stepping equation system, we demonstrate that this system has a unique solution and provide a solver that is guaranteed to find this solution in a finite number of steps from any initial guess. Comparisons with the Neumann analytical solution and tests in the Lund-Potsdam-Jena managed Land vegetation model reveal that the new method does not introduce significantly higher computational costs than the widely used DECP method while providing greater accuracy. In particular, it avoids a known nonphysical artifact in the solution.
\end{abstract}



\section{Introduction} 
Human activities have caused a transformation of all parts of the earth system, including atmosphere, land, oceans and cryosphere. The need of discerning and evaluating the past and potential future impacts of humans on their environment renders process-based computer models as an indispensable tool, since they allow for the consideration of counterfactual scenarios. Dynamic Global Vegetation Models (DGVMs) and Land Surface Models (LSMs) represent the biophysical water, energy and substance cycles that happen within land systems and in exchange with the other system parts.

While the governing laws of the involved processes come from natural sciences, mathematical modeling and numerical analysis is often necessary to formulate and solve the resulting sets of equations. In particular for cases of nonlinear dynamics, the standard methods of calculating numerical solutions are often either not applicable or too slow to be viable in long-term, global simulation runs. One example of such dynamics is the conduction of heat within the soil for regions where phase changes of the water in the soil plays a significant role in the ground thermal regime. In permafrost regions, which cover up to 22\% of the area of the northern hemisphere \cite{Obu2021}, the top layer of the soil usually thaws during summer and refreezes during winter.  Without considering the release and absorption of latent heat which makes the set of differential equations nonlinear, the temperature for this type of ground cannot be accurately simulated.

The numerical solution of the heat equation with phase change, commonly referred to as the Stefan problem, is a well-established and extensively studied topic in numerical analysis \cite{Crank1984, Rubenstein1971}. Numerous methods and results are available also for the enthalpy formulation adopted in this study \cite{Voller1981, Elliott1987, Nochetto1991, Dipietro2015}. In the field of geosciences, the enthalpy formulation is employed to model and numerically solve the energy balance in ice sheet simulations, as for example in the Parallel Ice Sheet Model (PISM) \cite{Aschwanden2012}.

However, the application of this formulation in DGVMs and LSMs presents unique challenges: (1) handling highly heterogeneous and discontinuous data for thermal soil properties due to the layered soil structure,  (2) ensuring that the method remains stable and robust while solving more than 60,000 soil columns with heterogeneous forcing and parameters in parallel over simulation periods exceeding 4,000 years (3) achieving minimal computational cost to enable such large-scale simulations.

Existing approaches implemented in DGVMs and LSMs vary in terms of the governing equations and discretization schemes employed. A comprehensive overview of the methods used in different software models is provided in Table 1 of \cite{Tubini2021}.

Most widely used is the decoupled energy conservation parametrization (DECP), wherein the heat conduction timestep uses the standard heat equation, while the effects of phase changes are incorporated afterwards by a correction that is based on an energy conservation argument \cite{CLM52018, Niu2011NoahMP, Ekici2014JSBACH, Schaphoff2018, Masson2021SURFEX, Nord2021, Wania2009}. Typically, the heat conduction timestep in DECP is solved using an implicit method such as Crank-Nicolson \cite{CLM52018, MPI2024, Masson2021SURFEX, He2023, Nord2021, Wania2009}. However, this method is known to suffer from an artificial stretch of the freezing region, which leads to inaccuracies \cite{Nicolskiy2007}. Another commonly employed approach to circumvent the nonlinearity in time stepping is the use of an apparent heat capacity that accounts for the phase-change effect \cite{Gouttevin2012SoilFreezing, Best2011JULES}. 

Methods that do solve the nonlinear time stepping system often lack a proven globally convergent solver for the time stepping system, which would be desirable for large-scale and long-term simulations  \cite{Tubini2021}.

To address this research gap, this study proposes a numerical method specifically designed for simulating soil temperatures in DGVMs and LSMs. The method is implemented and tested within the model Lund-Potsdam-Jena managed Land (LPJmL) which is a DGVM, to demonstrate its effectiveness for the intended application. The finite element method is based on the weak enthalpy formulation, which inherently incorporates phase changes, thereby eliminating the need for a posteriori corrections of numerical results. The time-stepping scheme is implicit, with both backward Euler and Crank–Nicolson methods developed and evaluated. The nonlinear solver leverages the piecewise-affine structure of the system, ensuring both computational efficiency and global convergence to the exact solution within a finite number of iterations. This solver is commonly known as the Katzenelson algorithm \cite{Katzenelson1965}.

We begin by presenting the weak enthalpy formulation of the Stefan problem, followed by the derivation of the discretized model in Section \ref{sec: Numerical method} and the corresponding system of nonlinear equations in Section \ref{sec: System of nonlinear equations}. In Section \ref{sec:Solving the implicit system}, we discuss theory on piecewise affine functions and tridiagonal matrices to prove that solutions to the system of equations are roots of a piecewise affine homeomorphism. This gives existence and uniqueness of discrete solutions as well as sufficient conditions for the Katzenelson solver to converge globally. Section \ref{sec: results} presents the results of a mesh convergence study and soil temperature simulations in LPJmL using the proposed method. Finally, Section \ref{sec:discussion} compares the new method to the DECP approach and discusses the key differences between the two.

For readers that are not familiar with the Stefan problem and its enthalpy formulation we recommend the introductory texts \cite{Alexiades1993} and \cite{Visintin2008}. An introduction to the finite element method that fits well with the terminology used in this paper can be found in \cite{Zulehner2008} and \cite{Zulehner2011}. Piecewise affine function theory is introduced in \cite{Scholtes2012}. 

\section{Mathematical model}  
LPJmL computes output variables on a spatial grid with a resolution of \(0.5^\circ\) in both latitude and longitude. As is standard practice in LSMs and DGVMs, the soil temperature within each grid cell is assumed to vary only with depth, while remaining uniform in the horizontal directions. Consequently, heat conduction is modeled as a one-dimensional process in space. The spacetime domain is defined as  
\begin{align*}
G := (0, D) \times (0, T),
\end{align*}
where $D$ denotes the maximum depth of the soil column and $T$ represents the total simulation duration. 

On this domain, the unknown temperature is denoted by $u(x, t)$. The given spatially dependent data include the frozen and unfrozen volumetric heat capacities, $c_f(x)$ and $c_u(x)$, respectively, as well as the frozen, mushy and unfrozen thermal conductivities, $k_f(x), k_m(x)$ and $k_u(x)$. Additionally, $L(x)$ represents the volumetric latent heat of fusion.

A Dirichlet boundary condition $s(t)$ is assigned at the surface and a zero flux Neumann boundary condition at the bottom. The initial condition is given as $u_0(x)$.

The heat conduction equation with phase change can be expressed either in terms of temperature or enthalpy. The classical two-phase Stefan problem uses temperature and assumes that the liquid and solid phases are separated by an interface with an unknown differentiable trajectory  
\begin{align*}
X: [0, T] \mapsto (0, D) \,.
\end{align*}

The spatial domain is then partitioned into two regions:  
\begin{equation}
\begin{aligned}
G^- &:= \{(x, t) \in G \mid x < X(t)\}, \label{eq:partition_of_G_Stefan} \\
G^+ &:= \{(x, t) \in G \mid x > X(t)\} \,.
\end{aligned}
\end{equation}
Within the subdomains $G^-$ and $G^+$, the standard heat equation holds. The principle of energy conservation, applied to the moving interface, leads to the Stefan condition at the free boundary, as expressed in equation (\ref{eq:Stefan_StefCond}).

\vspace{5pt}
\paragraph{Stefan problem  \cite{Alexiades1993}} {
Given boundary data $u_0(x)>0$ and $s(t)<0$ for all $(x,t)\in G$, find a curve $X(t)$ that divides $G$ in $G^-$ and $G^+$ as defined by (\ref{eq:partition_of_G_Stefan}) and a function $u(x,t)$ such that 
\begin{align}
     c_f(x) u(x,t)_t&=(k_f(x)u(x,t)_{x})_x  \quad &&(x,t)\in G^- \nonumber  \\ 
     c_u(x) u(x,t)_t&=(k_u(x)u(x,t)_{x})_x\quad &&(x,t)\in G^+ \nonumber \\[5pt]
     u(X(t),t) &= 0 \quad &&(0\leq t \leq T) \nonumber \\
     L(X(t)) X(t)_t &=k_f(X(t)) u(X(t)^-,t)_x -k_u(X(t)) u(X(t)^+,t)_x &&(0< t < T) \label{eq:Stefan_StefCond} \\[5pt]
     u(0,t) &= s(t), \quad  u(D,t)_x = 0 \quad &&(0\leq t \leq T) \nonumber\\
     u(x,0)&=u_0(x), ~~\,\quad X(0) = 0 &&(0<x<D) \nonumber.
\end{align}
}

A limitation of this formulation is that it can represent only a single frozen-unfrozen interface. This renders the model inadequate for permafrost soils, where one or more unfrozen ground layers may be entirely enclosed within an otherwise frozen soil column — a phenomenon known as \textit{closed talik}. 

The weak enthalpy formulation overcomes this limitation by using the volumetric enthalpy $e$ as the primary variable. The relationship between the temperature $u$ and the enthalpy $e$ is given by:
\begin{equation}
\label{eq:t(e)_acrossphase}
u=\beta(e,x):= 
\begin{cases}
\begin{alignedat}{3}
& \tfrac{e}{c_f(x)}    \quad &&\text{for } e\leq0     \quad &&\text{(frozen)} \\
&  0              \quad &&\text{for } 0<e<L(x)      \quad &&\text{(mushy)} \\
&  \tfrac{e-L(x)}{c_u(x)}  \quad &&\text{for } e\geq L(x)    \quad &&\text{(unfrozen)} \, . 
\end{alignedat}
\end{cases}
\end{equation}

By additionally defining the thermal conductivity as a piecewise function of temperature:  
\begin{align*}
k(u, x) := 
\begin{cases}
    k_u(x) \quad & \text{for } u > 0 \\
    k_m(x) \quad & \text{for } u = 0 \\
    k_f(x) \quad & \text{for } u < 0 \, ,
\end{cases}
\end{align*}
we can express the governing differential equation as:  
\begin{equation}
\label{heatcond_enthalpy}
e_t(x, t) = \bigl[k\bigl(u(x, t), x\bigr) u_x(x, t)\bigr]_x, 
\end{equation}
which holds for all $x$ within both phases $G^-$ and $G^+$, though not at the interface, and thus enables a unified description of the heat conduction process.

The weak formulation is derived from equation (\ref{heatcond_enthalpy}) by the standard procedure of multiplying both sides by test functions $\phi \in H^1(G)$, integrating over the spacetime domain $G$, and applying integration by parts to shift both temporal and spatial derivatives onto the test functions. Additionally, the space of admissible test functions is restricted to account for the essential boundary condition and the initial condition.

\vspace{5pt}
\paragraph{Weak Formulation \cite{Visintin2008}}{       
Given data $s\in L^\infty(0,T)$
and $e_0 \in L^\infty(0,D)$, find functions $e\in L^2(G)$ and $u\in L^2(0,T;\;H^1(0,D))$ with  $u(0,t)=s(t)$ for all $t\in [0,T]$ such that
\begin{align*}
u(x,t)=\beta(e(x,t),x) \quad \text{ for almost all $(x,t) \in G$ and }
\end{align*}
\begin{align} \label{eq:steafweak_main}
 \int\limits_{0}^T \int\limits_{0}^D [ e \phi_t  - k(u) u_x \, \phi_x ]\, dx \, dt =  - \int_0^D  e_0(x)  \phi(x,0) \, dx
\end{align}
for all  $\phi \in H^1(G)$ with $\phi(x,T)=0$ and $\phi(0,t)=0$.
}

\vspace{5pt}
The above formulation eliminates the need for an explicit partition of \(G\) into frozen and unfrozen regions. Hence, the number of interfaces is not predetermined but instead can emerge dynamically during the simulation, depending on the input data. In scenarios where only a single interface is present, it can be shown that this formulation is equivalent to the classical Stefan problem. In particular, the Stefan condition can be recovered from this formulation \cite{Alexiades1993}.

The existence and uniqueness of solutions to the weak enthalpy formulation, considering spatially dependent thermal conductivity but constant latent heat and heat capacity, have been proven in \cite{friedman1968stefan}.

\section{Discretization} \label{sec: Numerical method}
In the weak formulation, both the spatial and temporal derivatives are transferred to the test functions due to the lack of differentiability of $e$ and $u$ with respect to time. However, we do not employ a spacetime finite element method, because time-stepping is required to be a feature of our method. Hence, a time discretization is applied as the initial step, reducing the problem to a form that can be addressed using the standard finite element method.

\subsection{Time discretization}
We use an idea from Andreucci and Scarpa \cite{Andreucci2005LectureNO} in order to show that solutions of the weak formulation also approximately satisfy a time stepping problem.

Let ($u$, $e$) be a solution to the weak problem.  Fix a space-only test function $\varphi \in V_0=\{v \in H^1(0,D)| \, v(0)=0\}$ and let $t_1,\,t_2 \in [0,T]$ with $t_1<t_2$. For any $\epsilon>0$ with $\epsilon<\frac{t_2-t_1}{2}$ we define
\begin{align*}
\chi_\epsilon: (0,T) &\mapsto [0,1]\\ 
\chi_\epsilon(t) &:= \min\left(1, \, \tfrac{(t-t_1)_+}{\epsilon} ,\, \tfrac{(t_2-t)_+}{\epsilon}  \right)\, ,
\end{align*}
which belongs to $H^1(0,T)$ (see Figure \ref{fig:aux_funct}). 

\begin{figure}[ht]
\centering
    \includegraphics[width=70mm]{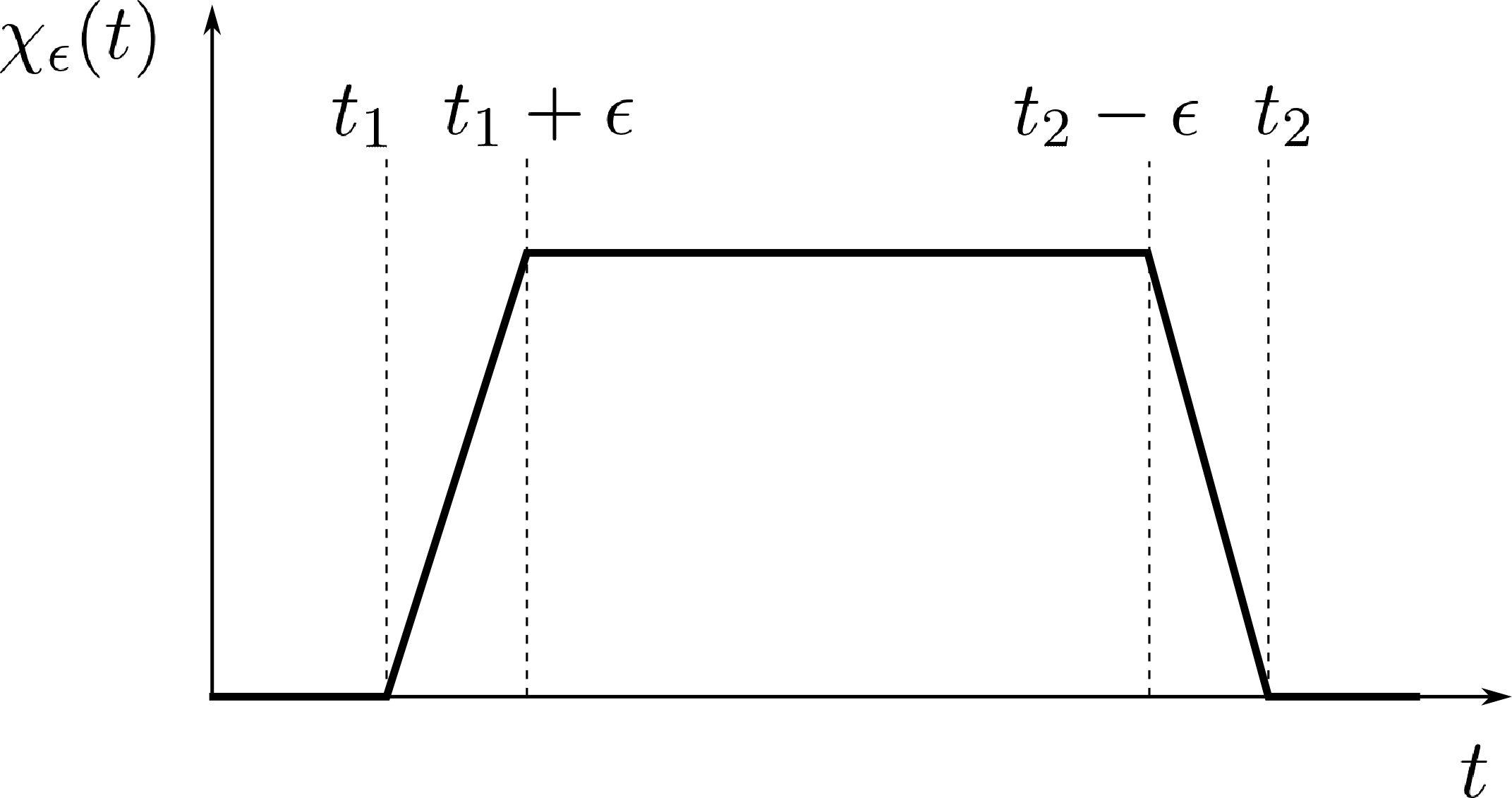}
    \caption{Illustration of the auxiliary function $\chi_\epsilon$.}
    \label{fig:aux_funct}
\end{figure}

The product $\varphi(x) \chi_\epsilon(t)$ is thus an element of $H^1(G)$ and an admissible test function of the weak formulation. Plugging it into (\ref{eq:steafweak_main}) the equation reads
\begin{eqnarray}    
\nonumber  \int\limits_{0}^T \chi_\epsilon(t)_t \int\limits_{0}^D e(x,t) \, \varphi(x)  \, dx \, dt - \int\limits_{0}^{T} \chi_\epsilon(t) \int\limits_{0}^D k(u(x,t),x) \, u(x,t)_x \, \varphi(x)_x  \, dx \, dt \\
 =  - \int_0^D   e_0(x)   \varphi(x)  \chi_\epsilon(0)  \, dx  \label{eq:weak_AltTestFnc} \,.
\end{eqnarray}
Since $\chi_\epsilon(0) = 0$, the right-hand side of (\ref{eq:weak_AltTestFnc}) vanishes. In addition the first left-hand side term can be expanded, yielding
\begin{eqnarray*}    
\nonumber \frac{1}{\epsilon}  \int\limits_{t_1}^{t_1+\epsilon} \int\limits_{0}^D e(x,t) \, \varphi(x)  \, dx \, dt -  \frac{1}{\epsilon} \int\limits_{t_2 -\epsilon}^{t_2} \int\limits_{0}^D e(x,t) \, \varphi(x)  \, dx \, dt  \\
 = \int\limits_{0}^{T} \chi_\epsilon(t) \int\limits_{0}^D k(u(x,t),x) \, u(x,t)_x \, \varphi(x)_x  \, dx \, dt \,.
\end{eqnarray*}
By the Lebesgue differentiation theorem the left-hand side limit $\epsilon \rightarrow 0$ exists for almost all $t_1$ and $t_2$. Taking the limit, it thus holds that
\begin{equation}
\begin{aligned} \label{eq:time_restricted_weak}
  \int_0^D e(x,t_1) & \, \varphi(x)  \, dx \, -  \int_0^D e(x,t_2) \, \varphi(x)  \, dx \\  &=    
   \int_{t_1}^{t_2} \int_0^D  k(u(x,t),x) \, u(x,t)_x \, \varphi(x)_x  \, dx \, dt \, 
\end{aligned}
\end{equation}
for almost all $t_1$, $t_2$ given our fixed test function $\varphi$.

Different time-stepping schemes can be obtained by applying quadrature rules to the integral of the right hand side of (\ref{eq:time_restricted_weak}). Choosing $\theta \in [0,1]$ we use the approximation
\begin{align*}
    \int_{t_1}^{t_2} f(t) dt \approx (t_2 -t_1) \big[ (1-\theta) \, f(t_1) + \theta \, f(t_2)\big]  \,,
\end{align*}
for an integrable function f. The cases $\theta = 0$, $\theta = 1$ and $\theta = \frac{1}{2}$ correspond to the \emph{left rectangle}, \emph{right rectangle}, and \emph{trapezoid} rules respectively.

For discrete times $t_0 =0,t_1, t_2,  \dots, t_N \in [0,T]$ we use the notations
\begin{align*}
e^n(x):=e(x,t_n), \quad u^n(x):=u(x,t_n), \quad s^n=s(t_n), \quad  \Delta t_{n}:= t_{n}-t_{n-1} .
\end{align*}
Based on the above argument, approximate solutions to the weak formulation can be obtained by solving this time stepping problem: 

\vspace{5pt}
\paragraph{Time Discrete Problem} {
Given $0=t_0<t_1< \dots < t_N=T \,$, $e^n \in L^2(0,D)$ and $u^n\in H^1(0,D)$ with $u^n(0)=s^n$ and $u^{n}(x)=\beta(e^{n}(x),\,x)$ almost everywhere,\\ 
find $e^{n+1} \in L^2(0,D)$ and $u^{n+1}\in H^1(0,D)$ such that
\begin{align*}
\left( \tfrac{ e^{n+1}- e^{n}}{\Delta t_{n+1}} , \varphi \right)_{L^2} = - \biggl( { (1-\theta)\, {\scriptstyle k(u^n)(u^n)_x} + \theta \, {\scriptstyle k(u^{n+1})(u^{n+1})_x }},~~ \varphi_x \biggr)_{L^2}
\end{align*}
for all  $\varphi \in V_0$ and in addition  
\begin{center}
$u^{n+1}(0)=s^{n+1}$ \quad and \quad  $u^{n+1}(x)=\beta(e^{n+1}(x),\,x)$  
\end{center}
almost everywhere.
}

\subsection{Space discretization}
To discretize the equations in space, we adapt a finite element method for the Stefan problem developed by Elliot \cite{Elliott1987}. The technique is modified to accommodate space-dependent physical parameters, a Neumann boundary condition, and various time-stepping schemes.

Consider a partition of the interval \((0, D)\) defined by the nodes $0 = x_0 < x_1 < \dots < x_{\Kh} = D$, where \(\Kh \in \mathbb{N}\). The subintervals 
\begin{align*}
T_k := (x_{k-1}, x_k), \quad k = 1, \dots, \Kh,
\end{align*}
are referred to as \emph{elements} and the collection $\mathcal{T}_h := \{ T_1, \dots, T_{\Kh} \}$ constitutes a \emph{subdivision} of $(0, D)$. The quantities
\begin{align*}
h_k := x_k - x_{k-1}, \quad k = 1, \dots, \Kh,
\end{align*}
are termed the \emph{element sizes}.
We use the finite dimensional space 
\begin{align*}
V_h := \{v\in C\big([0,D]\big)\big\vert \, v\vert_T \in P_1 \text{ for all } T\in \mathcal{T}_h \}\, ,
\end{align*}
where $P_1$ is the space of polynomials of degree $\leq 1$, as well as the test function subspace 
\begin{align*}
V_{0,h} := \{ v \in V_h \vert \, v(0)=0 \} \, .
\end{align*}

The set $\{ \varphi_0,\dots , \varphi_{\Kh}\} \subseteq V_h$, where the functions are defined by
\begin{align*}
\varphi_i(x_k)= \delta_{ik} \quad \text{ for } k= 0,\dots,\Kh
\end{align*}
and $\delta_{ij}$ is the Kronecker delta, then forms a basis of $V_h$.  Similarly, a basis of $V_{0,h}$ is given by $\{ \varphi_1, \dots , \varphi_{\Kh}\}$.

At this point, we consider the functions $e$ and $u$ to belong to the space $V_h$. This, however, introduces a complication because we still want to require the constitutive relation $\beta(e)=u$, where the function $\beta(e) \in \mathcal{L}_2([0,D])$ is defined by $x \mapsto \beta(e(x),x)$. In general $\beta(e) \notin V_h$, even when $e\in V_h$. The following definition, adapted from \cite{Ciarlet2002} rectifies this problem:

\vspace{5pt}
\begin{definition}
 The $V_h$-interpolation operator is given by (see Figure \ref{fig:int_operator})
    \begin{align*}
     \Pi_h:~ \mathcal{L}_2([0,D]) \longrightarrow V_h, \quad \quad  f \mapsto \sum_{i=0}^{\Kh} f(x_i) \varphi_i \,.
    \end{align*}
\end{definition}

\begin{figure}[ht]
\centering
    \includegraphics[width=60mm]{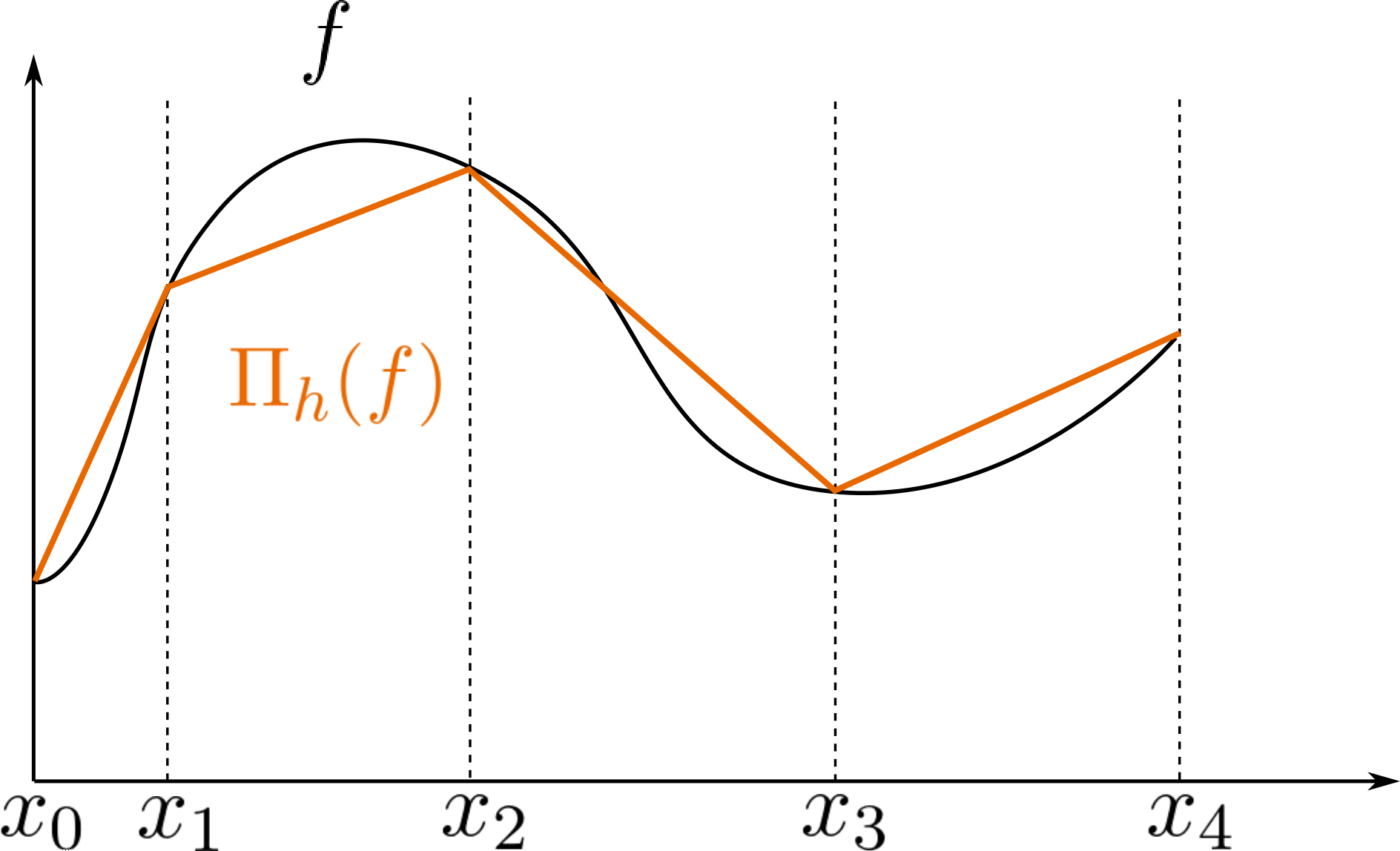}
    \caption{Illustration of the interpolation operator $\Pi_h$ applied to a function f.}
    \label{fig:int_operator}
\end{figure}

The function $ \Pi_h(f)$ is also characterized as the unique element in $V_h$ that takes the values of $f$ at all nodes. Now we can require $u =\Pi_h \circ \beta (e)$ which is consistent with the choices $u,e \in V_h$. The $\Pi_h$ operator is also used in the following Lemma/Definition. 

\vspace{5pt}
\begin{lemma}
By specifying
\begin{align*}
(f,\, g)_h:= \int_{0}^D \Pi_h(fg) dx
\end{align*}
for $f,g \in V_h$, we have defined an inner product on $V_h$.
\end{lemma}
\begin{proof}
Let $f,g\in V_h$. It holds that
\begin{align} \label{eq:discrete_inner_prod}
    (f,\, g)_h:=\sum_{i=1}^{\Kh} \int_{T_i} \Pi_h(fg) dx = \sum_{i=1}^{\Kh} \frac{h_i}{2} [f(x_{i-1}) g(x_{i-1}) + f(x_{i}) g(x_{i})],
\end{align}
where the last equality is due to the trapezoid rule being exact for $P_1$ functions. The bilinearity and symmetry directly follows. For strict positivity note that for a non-zero $f\in V_h$ there is a node $x_j$ with $f(x_j)\neq 0$. Hence, $(f,\,f)_h>0$. 
\end{proof}

\vspace{5pt}
\paragraph{Fully Discrete Problem} {
Given $e^n, u^n\in V_{h}$ with $u^n(0)=s^n$ and $u^{n}=\Pi_h\circ \beta(e^{n})$ , find $e^{n+1}, u^{n+1}\in V_h$ such that
\begin{align}
\left( \tfrac{e^{n+1}- e^{n}}{\Delta t_{n+1}} , \varphi \right)_{h}  = - \biggl( { (1-\theta)\, {\scriptstyle k(u^n)(u^n)_x} + \theta \, {\scriptstyle k(u^{n+1})(u^{n+1})_x }},~~ \varphi_x \biggr)_{L^2}\label{eq:space_time_discretized}
\end{align}
holds
for all  $\varphi \in V_{0,h}$ and in addition
\begin{align} \label{eq:fully_discr_beta}
\text{(i)} ~ u^{n+1}(0)=s^{n+1} \quad \text{ and } \quad \text{(ii)}~ u^{n+1}=\Pi_h \circ \beta(e^{n+1}) \, .
\end{align}
}

\vspace{5pt}
In the discrete setting the following \textit{data} is given. The thermal conductivities are assumed to be piecewise constant functions. Let 
\begin{align} \label{eq: ass_data_k}
k_i^u, k_i^m,k_i^f\ >0\text{ for all }i\in 1,\dots, \kappa \, .
\end{align}
We further define the functions
\begin{align} \label{eq:def_ki}
    k_i(u):=
    \begin{cases}
        k_i^f \quad & \text{ for } u<0 \\
        k_i^m \quad & \text{ for } u=0 \\
        k_i^u \quad & \text{ for } u>0 \, ,
    \end{cases}
\end{align}
for $i=1,\dots, \Kh$ and with that the function $k(u,x)$ by requiring for all $i\in 1, \dots, \kappa$ that
\begin{align*}
k(u,x)= k_i(u) \quad \text{for all } x\in T_i\ .
\end{align*}
For the other parameters we assume piecewise linearity. Let 
\begin{align} \label{eq: ass_data_stor}
c_i^f, ~ c_i^u, ~ L_i >0 \text{ for all } i\in 1,\dots, \kappa \, .
\end{align}
The functions $L(x), c_f(x), c_u(x) \in V_h$ are then defined by
\begin{align*}
c_f(x_i)=c_i^f,\quad c_u(x_i)=c_i^u, \quad L(x_i)=L_i \quad \text{ for all } i\in \{0,\dots,\Kh\} \,.
\end{align*}

\vspace{5pt}
\begin{remark}
    Equation (\ref{eq:discrete_inner_prod}) shows that using $(\cdot,\, \cdot)_h$ instead of $(\cdot,\, \cdot)_{L_2}$ amounts to employing the trapezoid rule to approximate the integral of the inner product. This procedure is one of the standard routes to obtain a diagonal, so called \emph{lumped} mass matrix in the system of equations, which has the desirable property of enabling an explicit time-stepping method. Indeed, we will obtain a diagonal mass matrix and an explicit scheme for $\theta=0$.
\end{remark}

\section{System of nonlinear equations} \label{sec: System of nonlinear equations}
In this section the Fully Discretized Problem will be converted to equivalent algebraic systems of nonlinear equations. 

Let $\tilde{\eta}^n \in \mathbb{R}^{\Kh+1}$ and $\gamma^n\in \mathbb{R}^{\Kh+1}$ be the coordinate vectors of $e^n$ and $u^n$ over the $V_h$-basis, such that
\begin{align*}    
   e^n = \sum_{j=0}^{\Kh} \tilde{\eta}_j^n \,\varphi_j
   \quad\text{and}\quad
   u^n = \sum_{j=0}^{\Kh} \gamma_j^n \,\varphi_j
   \quad\text{ for } n=1,\dots,N \,.
\end{align*}

The aim is to derive formulas for $\gamma^{n+1}$ and $\tilde{\eta}^{n+1}$, given $\gamma^n$ and $\tilde{\eta}^n$ such that the fully discrete problem is solved. Using the basis of $V_{0,h}$ and the linearity of the inner product, the statement that \eqref{eq:space_time_discretized} holds for all $\varphi \in V_{0,h}$  is equivalent to
\begin{equation} \label{eq:basis_i_fully_discr}
\begin{aligned}
    & \underbrace{\bigl( \tfrac{ e^{n+1}- e^{n}}{\Delta t_{n+1}}, \varphi_i \bigr)_h}_{\mathbf{A}}  = \\
& \quad \quad \quad - (1-\theta) \underbrace{\bigl( k(u^n)(u^n)_x, (\varphi_i)_x \bigr)_{L^2}}_{\mathbf{B}} - \theta \underbrace{\bigl( k(u^{n+1})(u^{n+1})_x, (\varphi_i)_x \bigr)_{L^2}}_{\mathbf{C}}  \,
\end{aligned}
\end{equation} 
being satisfied for all $\varphi_i$ with $i\in \{1,\dots,\kappa\}$.

We start with computing term $\mathbf{A}$. It holds that
\begin{align*}
\Bigl( \tfrac{e^{n+1}- e^{n}}{\Delta t_{n+1}}, \,\varphi_i \Bigr)_h
\;=\;
\Bigl(\sum_{j=0}^{\Kh}\tfrac{\tilde{\eta}_j^{n+1}-\tilde{\eta}_j^{n}}{\Delta t_{n+1}}
\,\varphi_j,\, \varphi_i \Bigr)_h
\;=\;
\tfrac{1}{\Delta t_{n+1}}
\sum_{j=0}^{\Kh} \bigl(\tilde{\eta}_j^{n+1}- \tilde{\eta}_j^{n}\bigr)
\;\bigl( \varphi_i,\varphi_j \bigr)_h.
\end{align*}
By defining the matrix \(\tilde{M}_h \in \mathbb{R}^{\Kh\times(\Kh+1)}\) via
\begin{align*}
   (\tilde{M}_h)_{ij}
   \;:=\;
   \bigl(\varphi_i,\varphi_j\bigr)_h,
   \quad
   \text{where } i=1,\dots,\Kh,\;j=0,\dots,\Kh,
\end{align*}
and denoting its $i$-th row by $\tilde{M}_h^i$, the last expression can be written as
\begin{align}
\label{eq:sol_to_disc_first_Scalprod}
\Bigl( \tfrac{e^{n+1} - e^n}{\Delta t_{n+1}},\;\varphi_i \Bigr)_h
\;=\;
\tfrac{1}{\Delta t_{n+1}}\;\tilde{M}_h^i\,
\bigl(\tilde{\eta}^{n+1} - \tilde{\eta}^{n}\bigr).
\end{align}
The matrix $\tilde{M}_h$ is called the \emph{mass matrix}. Its entries can be computed using \eqref{eq:discrete_inner_prod}:
\begin{align*}
   (\tilde{M}_h)_{ij}
   \;=\;
   \sum_{k=1}^{\Kh}\,\tfrac{h_k}{2}\;\bigl[\varphi_i(x_{k-1})\,\varphi_j(x_{k-1}) + \varphi_i(x_k)\,\varphi_j(x_k)\bigr]
   \;=\;
   \tfrac{\delta_{ij}}{2}
   \;\sum_{k=1}^{\Kh} h_k\,\bigl(\delta_{i,k-1} + \delta_{i,k}\bigr).
\end{align*}
As a result we obtain
\begin{eqnarray*}
   \tilde{M}_h = 
   \left[\begin{array}{c c}
0 &  \begin{matrix} 
\tfrac{h_1+h_2}{2} & 0                 & \cdots                  & 0    \\
0                 & \ddots            & \ddots                  & \vdots \\ 
\vdots            & \ddots            & \tfrac{h_{k-1}+h_{k}}{2} & 0 \\
0              &  \cdots            & 0                       & \tfrac{h_k}{2} 
\end{matrix}
   \end{array}\right] \, .
\end{eqnarray*}

Note that the first column of $\tilde{M}_h$ is zero, so the top entries of $\tilde{\eta}^{n+1}$ and $\tilde{\eta}^n$ do not contribute to (\ref{eq:sol_to_disc_first_Scalprod}). We define $M_h \in \mathbb R^{\kappa\times \kappa}$ to be $\tilde{M}_h$ with first column deleted and $\eta_n\in \mathbb R^{\kappa}$ to be $\tilde{\eta_n}$ with first entry deleted. The term $\mathbf{A}$ is then equal to
\begin{align*}
\tfrac{1}{{\Delta t_{n+1}}} M_h^i (\eta^{n+1} -\eta^{n})\,.
\end{align*}

\vspace{5pt}
\begin{remark}
When $s^n=0$ the first entry of $\tilde{\eta}^{n+1}$ is not uniquely defined by the boundary condition applied to $u$, as the preimage of $0$ under $\beta$ contains multiple elements. Indeed, this entry is not uniquely defined in that case, which is not a problem however, since it is not required to advance to the next time step.
\end{remark}

\vspace{5pt}
We continue with expanding the term $\mathbf{B}$. To this end, first let $j\in\{1,\dots,\Kh\}$ and consider $\big(k(u^n)(u^n)_x,~ ({\varphi_i})_x \big)_{L^2(T_j)},$ where the integral is restricted to the element $T_j$. Note that for $x\in T_j$ we have 
\begin{align*}
{\varphi_i}(x)_x=\frac{\delta_{j,i}-\delta_{j,i+1}}{h_j} \quad \text{ and } \quad u^n(x)_x = \frac{\gamma_j^n-\gamma_{j-1}^n}{h_j}\,. 
\end{align*}
Hence, we obtain
\begin{align}
    \big( k(u^n)(u^n)_x,~ ({\varphi_i})_x \big)_{L^2(T_j)} &= \int_{T_j}  k(u^n)(u^n)_x (\varphi_i)_x \, dx \nonumber  \\
    &= (\delta_{j,i}-\delta_{j,i+1}) \frac{ \gamma_j^n-\gamma_{j-1}^n }{h_j} \int_{T_j} \frac{k_j(u^n)}{h_j} dx \nonumber \\
    &=   (\delta_{j,i}-\delta_{j,i+1}) \frac{\gamma_{j}^n-\gamma_{j-1}^n }{h_j} \, \overline{k_{j}(u^n)}^j \label{eq:sol_to_disc_secScal}
\end{align}
where we have used the notation $\overline{f}^j$ for the average of a function $f$ over $T_j$. 
We define  
\begin{eqnarray}
\label{Def:Q}
  Q_j(\gamma^n):= \frac{\gamma_{j}^n-\gamma_{j-1}^n }{h_j} \, \overline{k_{j}(u^n)}^j ,
\end{eqnarray}
and show in Appendix \ref{App1} that 
\begin{eqnarray}
\label{Manip:Q}
  Q_j(\gamma^n) = \frac{k_j(\gamma_{j}^n)\gamma_{j}^n-k_j(\gamma_{j-1}^n)\gamma_{j-1}^n }{h_j}.
\end{eqnarray}

Finally, the term $\mathbf{B}$ is  computed as
\begin{align*}
    \big( k(u^n)(u^n)_x,~ ({\varphi_i})_x \big)_{L^2} &= \sum_{j=0}^{\Kh} \big( k(u^n)(u^n)_x,~ ({\varphi_i})_x \big)_{L^2(T_j)} \\
    &= \sum_{j=0}^{\Kh} (\delta_{j,i}-\delta_{j,i+1}) Q_j(\gamma^n) \, .
\end{align*}
Calculating the sum in the above expression, we arrive at
\begin{align}\label{eq:alg_F_def}
    \big( k(u^n)(u^n)_x,~ ({\varphi_i})_x \big)_{L^2} =  F_i(\gamma^n) := 
    \begin{cases}
        Q_{i}(\gamma^n)- Q_{i+1}(\gamma^n) \quad &\text{ for } i<\Kh \\ 
        Q_{i}(\gamma^n)                    \quad &\text{ for } i=\Kh \, .
    \end{cases} \, , 
\end{align}
where we have also defined the functions $F_i: \mathbb R^\Kh \rightarrow \mathbb R$.

The term $\mathbf{C}$ is calculated analogously to term $\mathbf{B}$.

Combining the results for the terms $\mathbf{A}$, $\mathbf{B}$ and $\mathbf{C}$, equation \eqref{eq:basis_i_fully_discr} reads
\begin{align*}
\tfrac{1}{{\Delta t_{n+1}}} M_h^i (\eta^{n+1} - \eta^{n}) = - (1-\theta) F_i(\gamma^n) - \theta F_i(\gamma^{n+1}) \,,
\end{align*}
which we require to hold for all $i\in\{1,\dots,\,\Kh\}$.
Pooling these equations together to a system of equations by using 
\begin{align*}
F(x):=\big[F_1(x) ~ \dots ~ F_\Kh(x)\big]^T \, ,
\end{align*}
we arrive at
\begin{align*}
  \tfrac{1}{\Delta t_{n+1}}  M_h (\eta^{n+1} - \eta^{n}) = - (1-\theta)  F(\gamma^n) - \theta F(\gamma^{n+1}) \,,
\end{align*}
which we have shown to be equivalent to (\ref{eq:space_time_discretized}) holding for all $\varphi \in V_{0,h}$. 

Now we incorporate the equation (\ref{eq:fully_discr_beta}). With the notation 
\begin{align*}
\B(\eta):= \big[\beta(\eta_1, x_1) ~ \beta(\eta_2, x_2) ~ \dots ~ \beta(\eta_{\Kh}, x_{\Kh}) \big]^T \,,
\end{align*}
the condition is equivalent to
\begin{align*}
    \gamma^{n+1} = \begin{bmatrix}s^{n+1}\\ \B(\eta^{n+1}) \end{bmatrix} \, . 
\end{align*}
To see this note that $\gamma_0^{n+1}=s^{n+1}$ by (i). Further the equation (\ref{eq:fully_discr_beta}) (ii) holds if and only if $u^{n+1}$ and $\B(\eta^{n+1})$ are equal at all nodes, by the characterization of $\Pi_h$. By defining the map
\begin{align} \label{eq:alg_F^n}
F^n: \mathbb{R}^\kappa \longrightarrow \mathbb{R}^\kappa, ~ x \mapsto \ F( \big[s^n ;~ x\big]) \, ,
\end{align}
we have $F(\gamma^n)= F^n\circ \B(\eta^n)$.  

\vspace{5pt}
\begin{proposition}
A solution of the Fully Discrete Problem is implicitly defined by the nonlinear system
\begin{align} \label{eq:fully_disc_implicit}
\tfrac{1}{{\Delta t_{n+1}}} M_h (\eta^{n+1} -\eta^{n}) = -(1-\theta) \big[F^n\circ \B (\eta^n)\big]  -\theta \big[F^{n+1}\circ \B (\eta^{n+1}) \big] \,.
\end{align}
\end{proposition}

\vspace{5pt}
\begin{remark} \label{expl:alg}
In case of $\theta = 0$ an explicit solution of the \textit{Fully Discrete Problem} is given by the time stepping algorithm
\begin{equation*}
\begin{aligned}
    \eta^{n+1}   &=\eta^n - \Delta t_{n+1} M_h^{-1}  F(\gamma^n)  \, \\
    \gamma^{n+1} &=  \big[s^{n+1};~ \B(\eta^{n+1})\big] \, ,
\end{aligned}
\end{equation*}
where semicolon is MATLAB notation to stack two column vectors on top of each other.
\end{remark}

\section{Solving the nonlinear system} \label{sec:Solving the implicit system} Since the system \eqref{eq:fully_disc_implicit} is nonlinear, no direct solution formula is available. In this section we employ theory on piecewise affine functions as well as matrix theory to prove existence and uniqueness of solutions and give a solution algorithm.

\subsection{Piecewise affine function theory}
Following Scholtes \cite{Scholtes2012} we define: 

\vspace{5pt}
\begin{definition}
    A continuous function $f\,:\, \mathbb R^n \longrightarrow \mathbb R^m$ is called \emph{piecewise affine} if there exists a finite set of functions $ f_i(x)=A^i\,x+b^i, ~ i=1,\dots,k$, such that $f(x)\in\{f_1(x),\dots,f_k(x)\}$ for every $x\in \mathbb R^n$. The set of pairs $(A^i,\,b^i), ~ i=1,\dots,k$ is called a collection of \emph{matrix vector pairs} corresponding to f.
\end{definition}
\vspace{5pt}

Our main result of this section is based on the below theorem of Fujisawa and Kuh \cite{Fujisawa1972}. Recall that a homeomorphism is a continuous bijective map with a continuous inverse. 

\vspace{5pt}
\begin{theorem} \label{th:pieceAffHomeo2}
    Let $f: \mathbb R^n \rightarrow \mathbb R^n$ be piecewise affine and $(A^1,\,b^1), \dots, (A^k,\,b^k)$ a collection of matrix vector pairs corresponding to $f$. We denote by $A_l^i$ the submatrix of $A^i$ composed of its first l columns and rows. If for each $l=1,\dots,n$ the determinants of
    \begin{align*}
    A_l^1, \dots , A_l^k
    \end{align*}
    do not vanish and have the same sign, then $f$ is a homeomorphism.
\end{theorem}

\vspace{5pt}
To apply the theorem the following result from Gantmacher and Klein \cite{Gantmacher2002} about tridiagonal matrices of the form
\begin{align} \label{eq:tridiagform}
T=\begin{bmatrix}
a_1 & b_1 &  &  & 0 \\
c_1 & a_2 & \ddots &  &  \\
 & c_2 & \ddots & b_{n-2} &  \\
 &  & \ddots & a_{n-1} & b_{n-1} \\
0 &  &  & c_{n-1} & a_n 
\end{bmatrix}   \,,
\end{align}
is needed.

\vspace{5pt}
\begin{theorem} \label{th:realTridiag2}
    Let $T$ be tridiagonal with the property 
    \begin{align*}
    b_i,\, c_i<0 \quad \text{ for all }i\in \{1,\dots,n-1\}
    \end{align*}
where $b_i$, $c_i$ are defined as in (\ref{eq:tridiagform}). Then, all eigenvalues of $T$ are real and simple. 
\end{theorem}
\vspace{5pt}

In addition a simple consequence of the Gershgorin circle theorem will be employed:
\vspace{5pt}
\begin{corollary}  \label{cor:eigvalDiagdom2}
    Let $A\in \mathbb R^{n\times n}$ with $a_{ii}> 0$ for all $i\in \{1,\dots,n\}$ be strictly diagonally dominant. Then, for any eigenvalue $\lambda$ of A it holds that $\Re(\lambda)> 0$.
\end{corollary}
\vspace{5pt}

In combination Theorem \ref{th:realTridiag2} and Corollary \ref{cor:eigvalDiagdom2} yield the following result.
\vspace{5pt}
\begin{proposition} \label{prop:specialTriag}
Let $T$ be a tridiagonal matrix of form (\ref{eq:tridiagform}) such that the conditions of Corollary \ref{cor:eigvalDiagdom2} hold and for each column $i$ of $T$ all off-diagonal entries either jointly vanish or are jointly smaller than zero. Then all leading principle minors of $T$ are strictly positive.
\end{proposition}
\begin{proof}
Let $\hat{T}$ be a submatrix of $T$ composed of its first $k$ columns and rows. We need to show that $\det(\hat{T})>0$. All conditions that we assumed to hold for $T$ also hold for $\hat{T}$. 

The first step of the prove is to reduce the problem to the case $c_i,b_i<0$ for all $i\in\{1,\dots,n\}$. Assume there exists an $i\in\{1,\dots,n\}$ such that the off-diagonal entries of the $i-th$ column vanish. Then, $\hat{T}$ has the form
\begin{figure}[ht]
    \centering
    \includegraphics[width=42mm,trim={0.2cm 1cm 1cm 0.2cm}, clip]{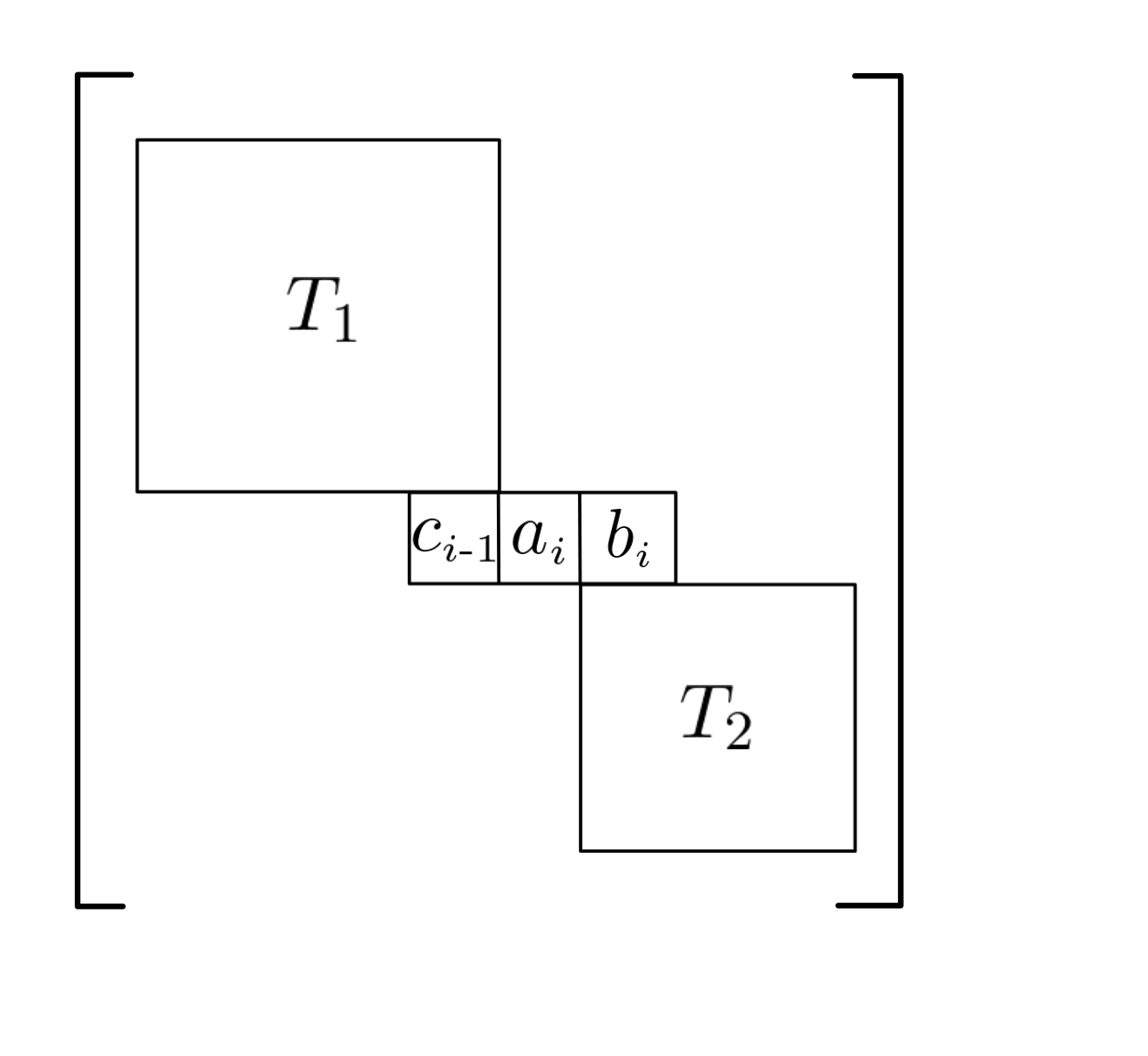} \, 
\end{figure}
\noindent where $T_1$ and $T_2$ are square tridiagonal matrices for which again the assumptions on $T$ hold. Recall that any eigenvalue of $T$ is a complex root of $\rho_{\hat{T}}(\lambda)=\det(\hat{T}-\lambda I)$. By the above form of $\hat{T}$, expanding this determinant at column $i$ with the Laplace expansion formula yields 
\begin{align*}
\rho_{\hat{T}}(\lambda) = \det(T_1-\lambda I) \det(T_2-\lambda I) (a_{ii}-\lambda)\, .
\end{align*}
Hence, the spectrum decomposition 
\begin{align*}
\lambda(\hat{T})=\lambda(T_1)\cup \lambda(T_2) \cup \{a_{ii}\}
\end{align*}
holds. Now we apply the same decomposition to $T_1$ and $T_2$ respectively, if columns with vanishing off-diagonal entries exist in those matrices and continue recursively.

It remains to show the case $c_i,b_i<0$ for all $i\in\{1,\dots,n\}$. Theorem \ref{th:realTridiag2} implies that all eigenvalues of $\hat{T}$ are real and Corollary \ref{cor:eigvalDiagdom2} further establishes strict positivity of these eigenvalues.
Thus, it holds that
\begin{align*}
\det(\hat{T})=\prod_{\lambda \in \lambda(\hat{T})} \lambda \, > 0.
\end{align*}
\end{proof}

\subsection{Homeomorphism property of the implicit system} 
Consider the function
\begin{equation}
\begin{aligned} \label{def:Phi}
    \Phi: \mathbb R^\Kh &\rightarrow \mathbb R^\Kh \\ 
         x &\mapsto \tfrac{1}{\Delta t_{n+1}} M_h \, (x - \eta^n) + \theta F^{n+1}\circ \B (x) + (1-\theta) \big[F^n\circ \B (\eta^n)\big] \, .
\end{aligned}
\end{equation}
Being a solution of (\ref{eq:fully_disc_implicit}) is equivalent being a root of $\Phi$. Thus, to apply the above theory we need to establish that 
\begin{enumerate}[(i)]
    \item the function $\Phi$ is piecewise affine and \label{enum: (i)}
    \item the corresponding matrices fulfill the conditions of Proposition \ref{prop:specialTriag}. \label{enum: (ii)}
\end{enumerate}
Then by Proposition \ref{prop:specialTriag} the assumptions of Theorem \ref{th:pieceAffHomeo2} are fulfilled which proves that $\Phi$ is a homeomorphism.

The continuity of $\Phi$ is given, since the the function $\beta(\eta, x_i)$ is continuous in $\eta$ and for the term $k_i(\beta(\eta,x_i))\,\beta(\eta,x_i)$ in \eqref{Manip:Q} the one-sided limits $\eta \downarrow 0$ and $\eta \uparrow 0$ at $\eta=0$ are both zero. The same holds for the one sided limits at $\eta = L_i$. 

To show the remaining part of (i) as well as (ii) we explore the nonlinear term $F^{n+1}\circ \B$ of \eqref{def:Phi}. For that we partition $\mathbb{R}^{\Kh}$ into box-shaped polyhedra as follows. Consider the map
\begin{equation}
    \begin{aligned} \label{eq:prop_dom_part1}
q_i: \mathbb{R} &\longrightarrow \{-1,\,0,\,1\},  \\
           \eta &\mapsto \begin{cases}
    -1 &\text{ for } \eta< 0 \\
    0  &\text{ for } 0\leq \eta \leq L_i \\
    1 &\text{ for } L_i < \eta \, ,
    \end{cases}
\end{aligned}
\end{equation}
with $i\in \{1, \dots , \kappa\}$. In addition, we define the set
\begin{align} \label{eq:prop_dom_part2}
\mathcal{Z} := \{-1,0,1\}^{\Kh}\,, 
\end{align}
and declare for each element $(z_1,\dots ,z_{\Kh})\in \mathcal{Z}$ the polyhedron
\begin{align} \label{eq:prop_dom_part3}
 P_z := \bigotimes_{i=1}^{\Kh} q_i^{-1}(z_i) \subseteq \mathbb R^\Kh \, .
\end{align}
For example, it holds that $P_{(1,0,1)}=(L_1,\infty)\times [0,L_2] \times (L_3,\infty)$  (see also Figure \ref{fig:ill_partition}). 
\begin{figure}
    \centering
    \includegraphics[width=70mm]{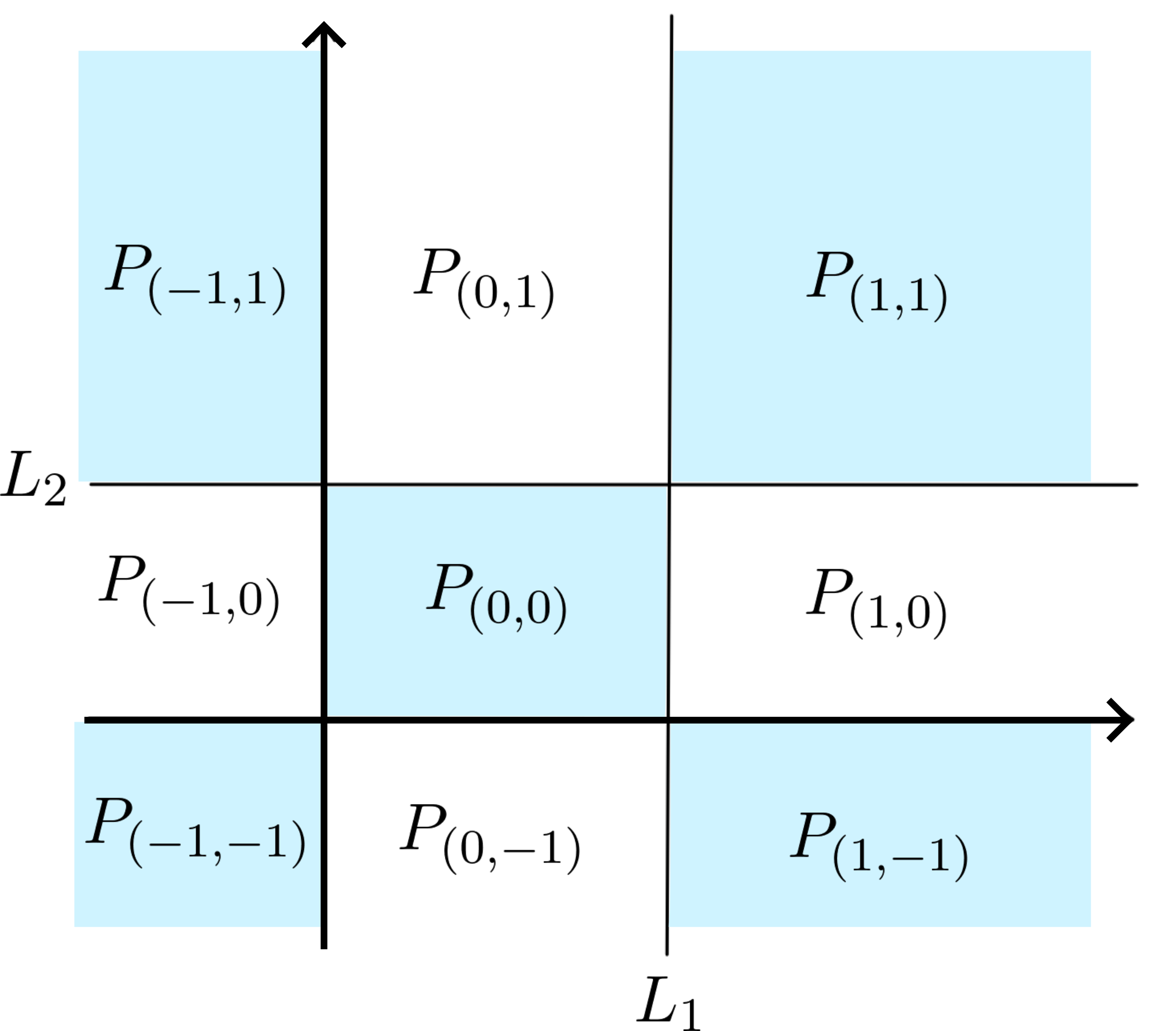}
    \caption{Illustration of the partitioning of $\mathbb R^2$, by (\ref{eq:prop_dom_part1}), (\ref{eq:prop_dom_part2}) and (\ref{eq:prop_dom_part3}).}
    \label{fig:ill_partition}
\end{figure}
It can easily be checked that $\big( P_z \big)_{z\in \mathcal{Z}}$ is a partition of $\mathbb R^\kappa$.

To simplify notation in the matrix definition \eqref{eq:def:A(z)} of the below theorem, instead of placing $z_i$ as a superscript on each variable of the $i-th$ column for all $i\in \{1,\dots,\kappa\}$, the common superscript is written above each column. In addition $\1_{x=y}$ denotes the indicator function $\1_A(x)$ on the set $A:=\{y\}$.

\vspace{5pt}
\begin{theorem} \label{th:pieceAffFnbeta}
For any $z\in\mathcal{Z}$ it holds that
\begin{align*}
F^n\circ \B(\eta)= A(z)\big[\eta - b(z)\big]+c^n \quad \quad \text{ for all } \eta \in P_z,
\end{align*}
where 
\begin{align} \label{eq:def:A(z)}
A(z):= 
\begin{bmatrix}
\colind{{\scriptstyle[r_1+r_2 ]}\,g_1}{z_1}&\colind{-r_2\,g_2}{z_2}&&\colind{\textcolor{white}{0}}{z_{\Kh-1}}& \colind{0}{z_\Kh} \\
-r_2       \, g_1  & {\scriptstyle[r_2+r_3 ]}\, g_2& \ddots &  &       \\
                   &    -r_3    \,g_2& \ddots & -r_{\Kh-1}\,g_{\Kh-1}  &        \\
                   &               & \ddots & {\scriptstyle[r_{\Kh-1}+r_\Kh ]}\, g_{\Kh-1}  & -r_\Kh \,g_{\Kh}  \\
0                  &               &        &-r_\Kh\, g_{\Kh-1} & {\scriptstyle r_\Kh}\,g_\Kh \\
\end{bmatrix}\, ,
\end{align} 
\begin{align*}
    r_i^{z}:=\frac{k_i(z)}{ h_i}, \quad g_i^{z}:=\frac{\1_{z=-1}}{c_i^f}+\frac{\1_{z=1}}{c_i^u},\quad \text{\normalfont for } i=1,\dots,\Kh \,, ~ z\in\{-1,\, 0,\, 1\}\, ,
\end{align*}
\begin{align*}
b(z):=  [\1_{z_1=1} L_1,\dots  ,\1_{z_\Kh=1} L_\Kh]^T
\end{align*}
and 
\begin{align*}
c^n:=[ r_1^{\sgn(s^n)} s^n, 0, \dots, 0 ]^T \, . 
\end{align*}
\end{theorem}
\begin{proof}
Let $z\in\mathcal{Z}$ and $\mathbf{\eta}\in P_z$. Additionally, we define $\gamma:=\B(\eta)$. 

We begin with computing a matrix form of $\B$. Let $i\in\{1, \dots, \kappa\}$. By $\eta\in P_z$ we have $q_i(\eta_i) = z_i$ and thus it holds that
\begin{align*}
z_i=1  \Leftrightarrow \eta_i>L_i \text{ and } z_i=-1 \Leftrightarrow \eta_i<0 \, .
\end{align*}
Recalling (\ref{eq:t(e)_acrossphase}) we obtain 
\begin{align*}
\B(\eta)_i &=\beta(\eta_i, \, x_i)= \1_{\eta_i<0}\, \frac{\eta_i}{c_i^f}+ \1_{\eta_i>L_i}\, \frac{\eta_i -L_i}{c_i^u}  \\
&=\1_{z_i=-1}\, \frac{\eta_i}{c_i^f}+ \1_{z_i=1}\, \frac{\eta_i -L_i}{c_i^u} =\1_{z_i=-1}\, \frac{\eta_i-  \1_{z_i=1} L_i}{c_i^f}+ \1_{z_i=1}\, \frac{\eta_i -  \1_{z_i=1} L_i}{c_i^u} \\
&= \left[ \frac{\1_{z_i=-1}}{c_i^f}+ \frac{\1_{z_i=1}}{c_i^u}\right] \big[ \eta_i-  \1_{z_i=1} L_i\big]  = g_i^{z_i} [\eta_i - b(z)_i] \,.
\end{align*}
It follows that
\begin{align*}
\B(\eta)
= \diag(g_1^{z_1}, \dots , g_\Kh^{z_\Kh}) \big[\eta - b(z)\big] \,.
\end{align*}

To achieve a matrix form of $F^n$ we recall (\ref{eq:alg_F^n}) and (\ref{eq:alg_F_def}) and calculate
\begin{align*}
    F^n(\gamma) &=  F([s^n;\, \gamma ]) = \begin{bmatrix}Q_1([s^n;\, \gamma ])- Q_2([s^n;\, \gamma ]))\\ \vdots \\ Q_{n-1}([s^n;\, \gamma ]) - Q_{\Kh}([s^n;\, \gamma ]) \\Q_{\Kh}([s^n;\, \gamma ]) \end{bmatrix} =
\overbrace{
\begin{bmatrix}
1 & -1 &  & 0 \\
 & \ddots & \ddots &  \\
 &  & \ddots & -1 \\
0 &  &  & 1 \\
\end{bmatrix}
}^{:=\mathbb{I}}
\begin{bmatrix}
    Q_1([s^n;\, \gamma ]) \\ \vdots \\Q_{\Kh}([s^n;\, \gamma ])
\end{bmatrix}\,.
\end{align*}
\vspace{6pt}
Recalling \eqref{Manip:Q} and noting that $\tfrac{k_i(\gamma_j)}{h_i} = r_i^{z_j}$ the expression reads
\begin{align*}
\mathbb{I}~ \times \,
\begin{bmatrix}
    \tfrac{k_1(\gamma_1)\gamma_1- k_1(s^n)s^n}{h_1} \\
    \tfrac{k_2(\gamma_2)\gamma_2- k_2(\gamma_1)\gamma_1}{h_2} \\
    \vdots  \\
    \tfrac{k_{\Kh}(\gamma_{\Kh})\gamma_\Kh- k_{\Kh}(\gamma_{\Kh-1})\gamma_{\Kh-1}}{h_\Kh}\\
\end{bmatrix} &=
\mathbb{I} ~ \times \,
\begin{bmatrix}
r_1^{z_1} &  &  & 0 \\
-r_2^{z_1}  & r_2^{z_2} &  &  \\
  & \ddots & \ddots &  \\
0 &  &  -r_\Kh^{z_{\Kh-1}} & r_\Kh^{z_\Kh} \\
\end{bmatrix}
\begin{bmatrix}
\gamma_1 \\
\gamma_2 \\ 
\vdots \\ 
\gamma_\Kh
\end{bmatrix}
- c^n \\[18pt] &=  
\begin{bmatrix}
\colind{\scriptstyle r_1+ r_2}{z^1} &\colind{-r_2}{z^2} &\colind{\textcolor{white}{0} }{ \cdots}  & \colind{0}{z^\Kh} \\
-r_2  & \scriptstyle r_2 + r_3 &\ddots  &  \\
  & \ddots & \ddots & -r_\Kh \\
0 &  &  -r_\Kh & \scriptstyle r_\Kh \\
\end{bmatrix}
\begin{bmatrix}
\gamma_1 \\
\gamma_2 \\ 
\vdots \\ 
\gamma_\Kh
\end{bmatrix}
- c^n \, 
\end{align*}

The result follows from plugging in $\gamma = \B(\eta)$.
\end{proof}
\vspace{5pt}

Since the other terms of \eqref{def:Phi} are linear it follows that $\Phi$ is piecewise-linear with corresponding matrices 
\begin{align*}
    J_z:=  \tfrac{1}{\Delta t_{n+1}} M_h + \theta A(z) \, ,
\end{align*}
and thus \eqref{enum: (i)} is established. To conclude the prove of this subsection it remains to show the property \eqref{enum: (ii)}. 

The off-diagonal entries of the i-th column of $J_z$ are $-\theta\, r_i^{z_i}g_i^{z_i}$, $-\theta\,r_{i+1}^{z_i}g_i^{z_i}$ or both. Hence, if $z_i\in\{-1,1\}$ by the assumption on the data (\ref{eq: ass_data_k}) and (\ref{eq: ass_data_stor}) these values are jointly negative. On the other hand if $z_i=0$ these entries are jointly zero. 

In addition, the diagonal of $J_z$ is
\begin{align*}
({J_z})_{ii}=(M_h)_{ii}/\Delta t_{n+1} + \theta\,r_i^{z_i}g_i^{z_i}+\theta\,r_{i+1}^{z_i}g_i^{z_i}\,.
\end{align*}
which is strictly positive and strictly larger than the modulus of the off diagonal entries. Thus, we have proven \eqref{enum: (ii)} and it follows the main result of this section.

\vspace{5pt}
\begin{theorem} \label{th:implUniqueSol}
The map $\Phi$ as defined in \eqref{def:Phi} is a homeomorphism. In particular (\ref{eq:fully_disc_implicit}) has a unique solution $\eta^{n+1}$.
\end{theorem}

\subsection{Katzenelson algorithm for the implicit system}
We apply the Katzenelson Algorithm \cite{Katzenelson1965} to find a root of $\Phi$. Fujisawa and Kuh \cite{Fujisawa1972} prove that this algorithm finds the exact root of a function in a finite number of steps from any initial guess, provided that the function is a piecewise affine homeomorphism. By Theorem \ref{th:implUniqueSol} the latter is the case for $\Phi$ as defined in \eqref{def:Phi} and hence the Katzenelson algorithm is well suited for solving the implicit system. The authors also provide a cost efficient variant of the algorithm that we use. The idea of their method and prove is provided here for the readers convenience.

\vspace{5pt}
Take an initial guess $x^{(0)}$ in the interior of  $
P_{z^{(0)}}$, where $z^{(0)} \in \mathcal{Z}$. Calculate the Newton increment 
\begin{align*}
v^{(0)} =- J_{z^{(0)}}^{-1} \Phi(x^{(0)}).
\end{align*}
If $x^{(0)}+v^{(0)} \in P_{z^{(0)}}$, then 
\begin{align*}
\Phi(x^{(0)}+v^{(0)}) = \Phi(x^{(0)})+ J_{z^{(0)}} v^{(0)} = \Phi(x^{(0)}) - J_{z^{(0)}} J_{z^{(0)}}^{-1}\Phi(x^{(0)}) = 0
\end{align*}
so the solution is found. Otherwise, let 
\begin{align*}
\lambda^{(0)} := \max(\{\lambda >0 | x^{(0)}+\lambda v \in \overline{P_{z^{(0)}}}  \})
\end{align*}
and set 
\begin{align*}
x^{(1)} :=x^{(0)}+\lambda^{(0)} v^{(0)}\, .
\end{align*}
Then $x^{(1)}$ lies on the boundary of at least one other polygon $P_{z^{(1)}}$. Assume for now that no polygon corner is hit so the adjacent polygon is unique.

Next calculate the new Newton increment based on the Jacobian valid on $P_{z^{(1)}}$. That is, $v^{(1)} =- J_{z^{(1)}}^{-1} \Phi(x^{(1)})$. If $x^{(1)}+v^{(1)} \in P_{z^{(1)}}$, then the solution is again found. Otherwise, continue with choosing $\lambda^{(1)}$ as described above.

\begin{figure}[ht]
    \centering
    \includegraphics[width=0.3\linewidth]{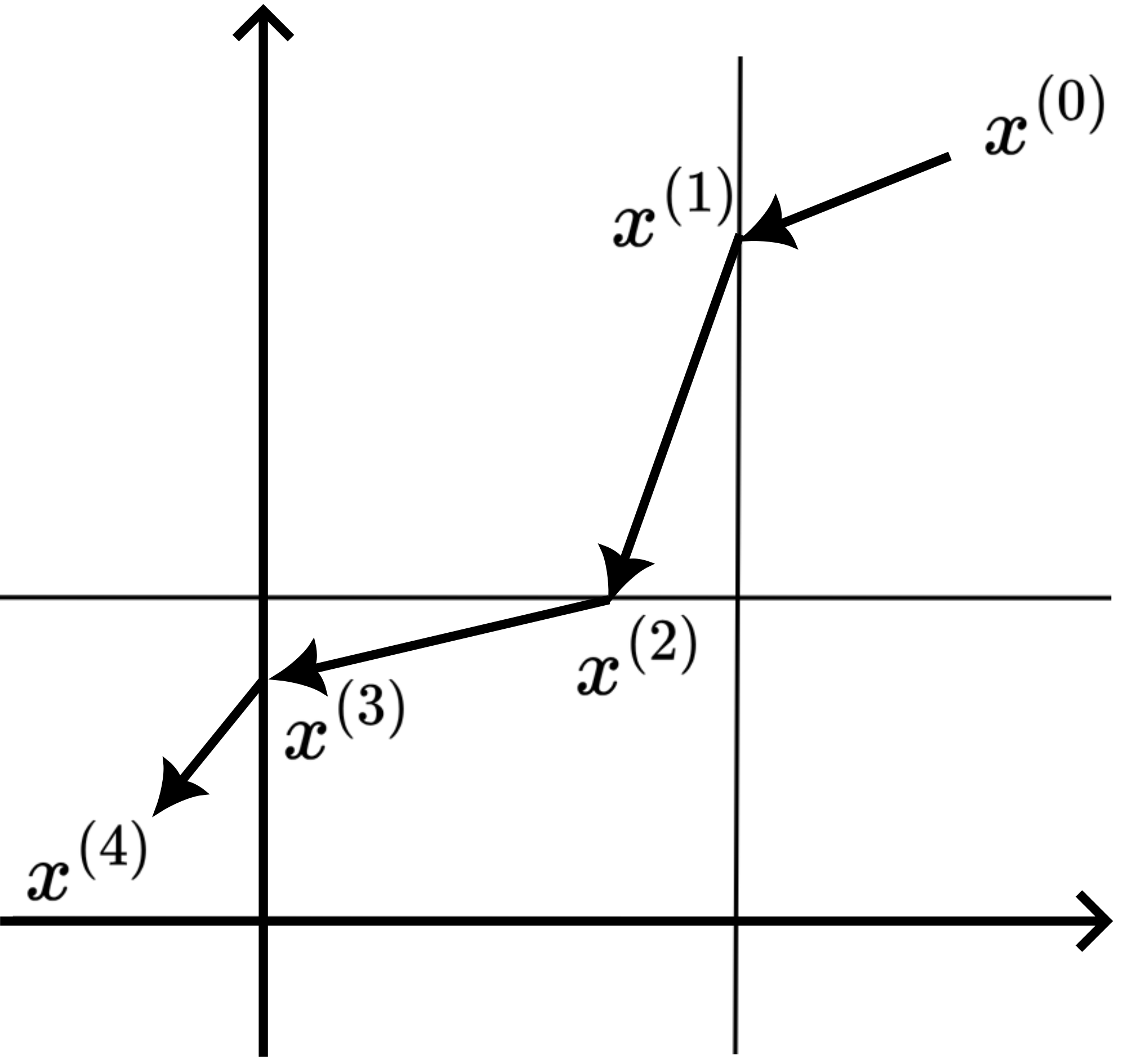}
    \caption{Example iterations of the Katzenelson algorithm. See also Figure \ref{fig:ill_partition}.}
\end{figure}

To see why this algorithm converges in a finite number of steps consider $L\subset \mathbb{R}^\kappa$ the line segment joining $\Phi(x^{(0)})\in \mathbb{R}^\kappa$ and $0\in \mathbb{R}^\kappa$. Since $\Phi$ is a homeomorphism the preimage $\Phi^{-1}(L)$ is a polygonal chain called "solution curve" in $\mathbb{R}^\kappa$ with vertices lying on the the boundaries of the $P_z$. It is easy to see that the above algorithm traces the solution curve exactly if no corner is hit.

If a corner is hit, the task is to determine an adjacent polyhedron to which the solution curve can be extended. Following Fujisawa and Kuh, we then apply a small enough perturbation to the solution curve to avoid hitting the corner and accomplish finding an extension. 

See Appendix \ref{App2} for the pseudo-code of the algorithm adapted to finding a root of $\Phi$ as defined in \eqref{def:Phi}. 

\section{Results} \label{sec: results}
The accuracy of the backward Euler method ($\theta = 1$), the Crank-Nicolson scheme ($\theta = \frac{1}{2}$), and the forward Euler method ($\theta = 0$) was evaluated by comparing numerical results with the Neumann analytical solution \cite{tarzia2015}. The Neumann solution is valid for the Stefan problem on a semi-infinite, homogeneous, one-dimensional domain with a Dirichlet boundary condition. The physical parameters were selected to be representative of typical conditions in permafrost soils. For all three numerical schemes, the computed results closely matched the analytical solution, provided that a sufficiently refined spatial mesh and a suitably small time step were used.

The results of the implicit method for increasing spatial resolution are presented in Figure \ref{fig:convergence}. Additionally, Figure \ref{fig:error} depicts the convergence of the numerical error toward zero as the mesh is refined.

\begin{figure}[ht]
        \centering
        \includegraphics[width=1\linewidth]{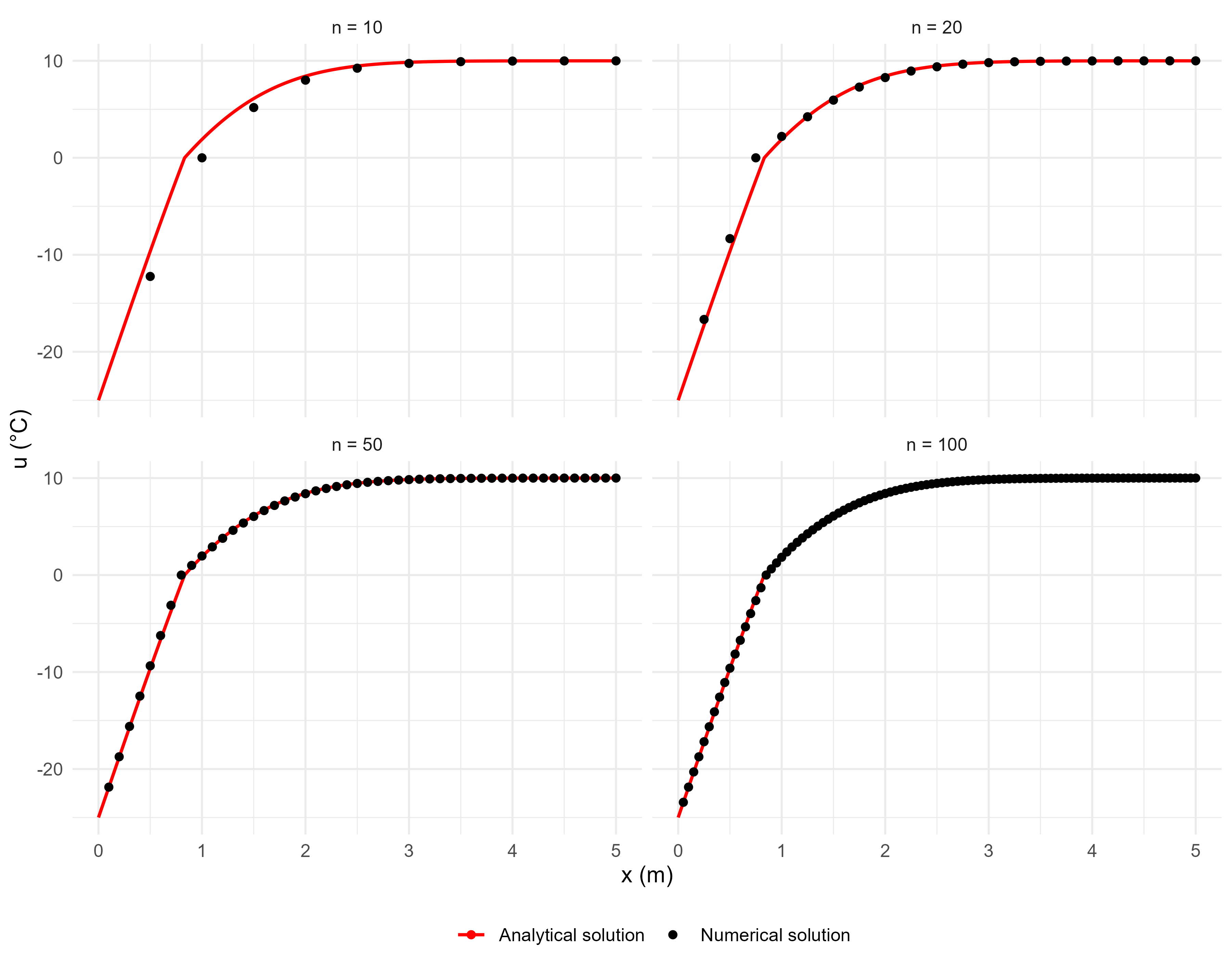}
        \caption{Convergence of numerical results to analytical temperature distribution at day 15 with increasing resolution. Temperature is denoted by $u$.}
        \label{fig:convergence}
\end{figure}

\begin{figure}[ht]
    \centering
    \includegraphics[width=0.55\linewidth]{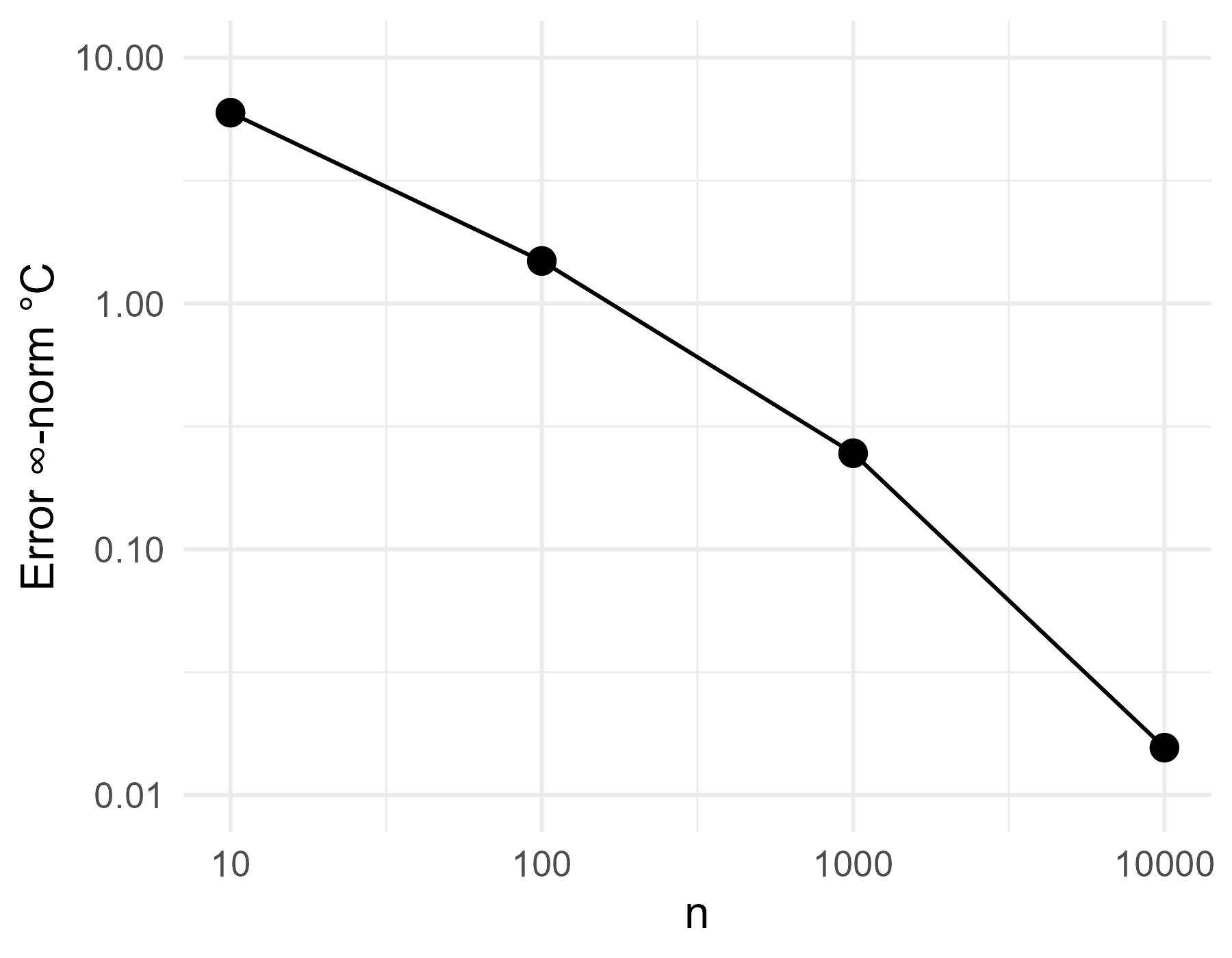}
    \caption{Maximum error in temperature distribution for all 20 days with increasing resolution.}
    \label{fig:error}
\end{figure}

The backward Euler method ($\theta = 1$) and the Crank–Nicolson scheme ($\theta = 0.5$), in combination with the Katzenelson solver, were implemented in LPJmL 5 \cite{Bloh2018} to assess their robustness and stability. A 13-meter-deep soil column was discretized for each of the more than $60,000$ grid cells using an unequally spaced grid that becomes progressively coarser towards the bottom ($n = 24$). The model uses a daily time step. As climate forcing the GSWP3-W5E5 dataset \cite{Lange2022Isimip} was used. The lower air temperature of this data served as the Dirichlet boundary condition. The thermal influence of a snow cover within a grid cell was accounted for in a simplified manner by dynamically increasing the thickness and decreasing the thermal conductivity of the top layer in the presence of snow. To optimize computational efficiency, the Crank–Nicolson scheme for the standard heat equation was applied whenever all temperatures in the soil column as well as the boundary condition had the same sign.

As expected, the Katzenelson method successfully converged for all grid cells throughout the entire simulation period of over 4000 years, for both the backward Euler and Crank–Nicolson schemes. Figure \ref{fig:soiltemp_maps} presents the resulting deep soil temperatures obtained using the backward Euler method, in comparison with observational data.

Since multiple iterations may be required to find the root of $\Phi$, the method is computationally more expensive than solving a linear heat equation. However, in the global LPJmL simulation, a single iteration was sufficient in most cases, indicating that no node changed state during the time step (see Figure \ref{fig:lin_sys_per_timestep} for the backward Euler method). During the first transient simulation year, an average of 1.48 and 1.93 linear system solutions per time step were required for the backward Euler and Crank–Nicolson schemes, respectively.

\begin{figure}[ht]
    \centering
    \includegraphics[width=1\linewidth]{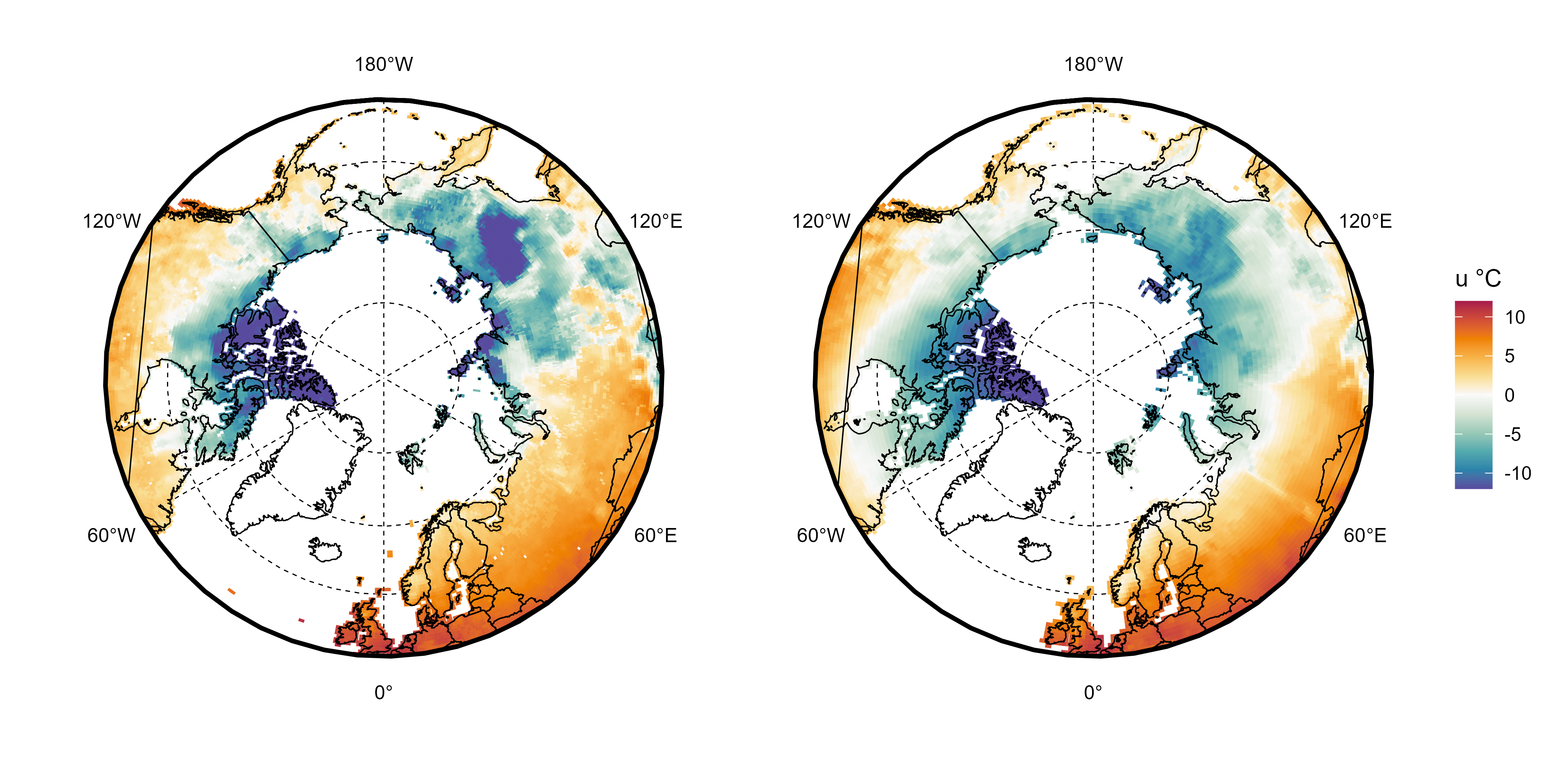}
    \caption{Mean soil temperature $u$ at the smallest soil depth that has practically zero fluctuations in ground temperature during the year (depth of zero annual amplitude) 2000-2014 of LPJmL (left) compared to statistically interpolated observations \cite{Aalto2018} (right).}
    \label{fig:soiltemp_maps}
\end{figure}

\begin{figure}[ht]
    \centering
    \includegraphics[width=0.55\linewidth]{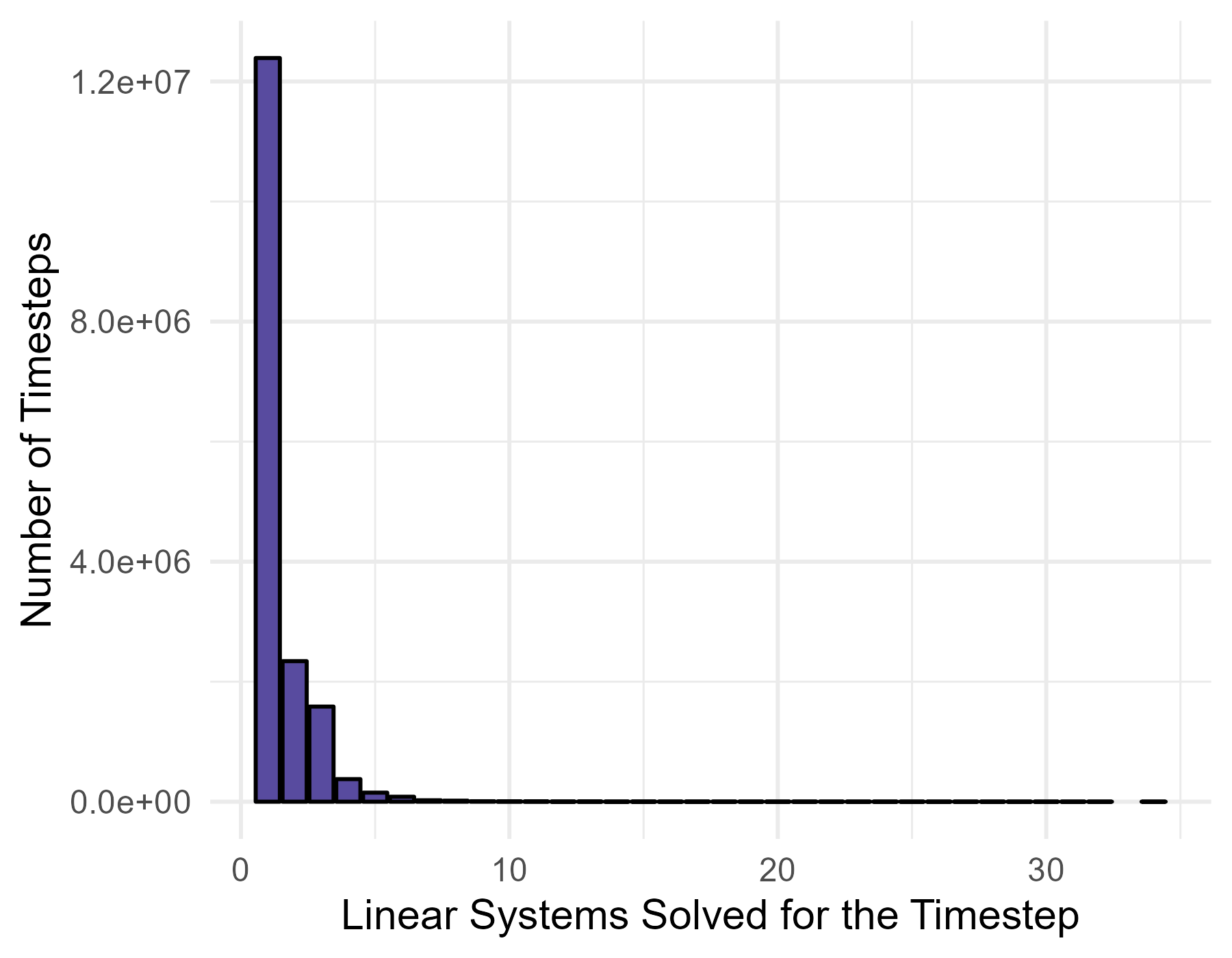}
    \caption{Number of timesteps (y-axis) that needed $1,2, 3,\dots,35$ linear systems solutions (x-axis) to carry out the step. The first year of the LPJmL simulation using the backward Euler method is shown.}
    \label{fig:lin_sys_per_timestep}
\end{figure}

\section{Discussion} \label{sec:discussion}

It is important to recall that the most commonly used implicit DECP method for soil heat transfer in LSMs and DGVMs neglects the nonlinearity in the partial differential equation during the heat diffusion time step. It only corrects for the effects of phase changes after the diffusion time step has been fully carried out. This approach has several implications when compared to the method proposed in this study.

First, the DECP method requires solving only a single linear system per time step, making it computationally more efficient than methods that handle a nonlinear system for time stepping. However, by leveraging the piecewise-linear structure of the system, the method proposed in this study typically requires solving only about $1.5$ to $2$ linear systems per time step, which remains computationally feasible for global-scale models. While the backward Euler method required fewer iterations per time step than the Crank–Nicolson scheme, the latter's higher spatial order of accuracy may still well justify its use.

Second, while the implicit DECP method is computationally more efficient, it is known to introduce an artificial stretch of the freezing region, leading to significant errors \cite{Nicolskiy2007}. This effect was reproduced in the present study (see Figure \ref{fig:DECP}). Over the full 20-day simulation period, the mean absolute difference between the analytical and numerical solutions was 0.443 for the DECP method, whereas the proposed method achieved a lower error of 0.170. Furthermore, it was observed that the explicit DECP scheme does not exhibit this nonphysical behavior and that the discrepancy increases with larger time steps.

\begin{figure}[ht]
    \centering
    \includegraphics[width=1\linewidth]{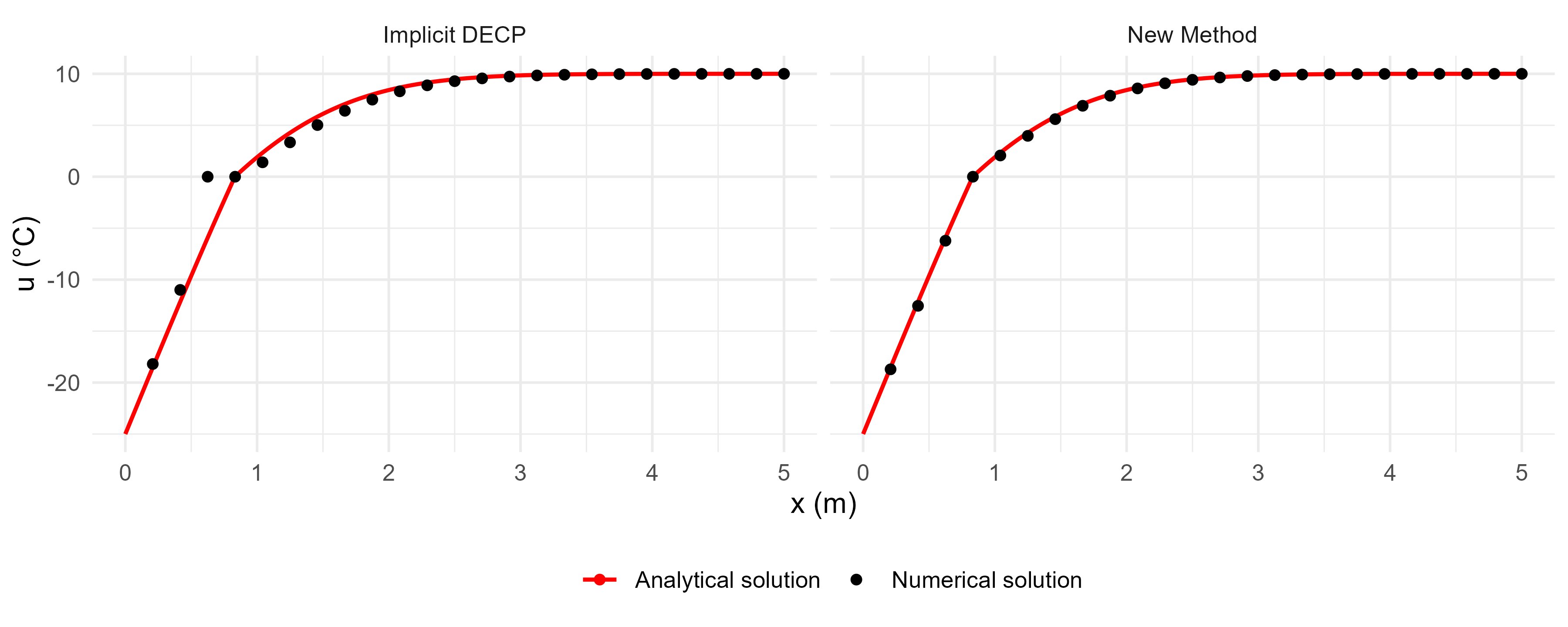}
    \caption{Implicit DECP method and new method proposed here plotted against analytical solutions at day 15. Temperature is denoted by $u$. Both methods used $\theta = \tfrac{1}{2}$. The time step was set to a day.}
    \label{fig:DECP}
\end{figure}

The errors in the implicit DECP method can be attributed to the fact that, in implicit schemes, energy fluxes are computed based on the temperature distribution of the next time step. Since the phase change parametrization is applied separately, the predicted temperature distribution is not yet adjusted for phase changes, leading to inaccuracies. This explanation is supported by the observation that the explicit DECP scheme does not exhibit this issue, as the fluxes in this approach are derived from the current temperature distribution, which has already been adjusted for phase changes.

A potential approach to enhance the computational performance of the implicit scheme developed in this study is to use a predictor-corrector method. In this approach, an explicit time step could be used to generate the initial guess for the Katzenelson algorithm. It is plausible that this initial guess would most of the times already lie within the same polyhedral region as the solution, thereby reducing the number of iterations needed to find the solution to one in those cases.

\section{Conclusion}
We proposed a numerical solver for one-dimensional heat conduction with phase changes targeted at its use in global DGVMs and LSMs. The solver is implicit, allowing it to accommodate large time steps while maintaining stability. The Katzenelson solution algorithm for the nonlinear timestepping system is proven to find the exact solution in a finite number of iterations. Comprehensive tests in the LPJmL model demonstrated that this robustness also holds in practice and that the computational costs remain comparable to those of the "industry standard" implicit DECP method; i.e., within a factor of two. 

However, the implicit method proposed in this study has the advantage of being consistently derived from fundamental physical principles, without erroneously assuming linearity. Additionally, it avoids computing fluxes based on a temperature distribution that has not yet been adjusted for phase changes, a limitation of the implicit DECP method that leads to larger errors and an artificial extension of the freezing region. In summary, the proposed method is better physically grounded and less error-prone while remaining computationally feasible for application in DGVMs and global LSMs.

\section*{Acknowledgments}
We thank the LPJmL developers Kirsten Thonicke, Werner von Bloh and Sibyll Schaphoff at the Potsdam Institute for Climate Impact Research (PIK) for support with implementing the newly developed method in LPJmL.

\appendix
\section{Derivation of (\ref{Manip:Q})}
\label{App1}
The expression 
\begin{align} \label{eq:def_algheatflux}
\frac{\gamma_{i}^n-\gamma_{i-1}^n }{h_i} \, \overline{k_{i}(u^n)}^i\,,
\end{align}
 in (\ref{Def:Q}) resembles the approximation to the \emph{negative} heat flux, as defined by Fourier's law.
It can be explored by computing
\begin{align}\label{eq:sol_to_discr_weight}
     \overline{k_i(u^n)}^i= k_i(\gamma_{i}^n) \, \omega + k_i(\gamma_{i-1}^n) \, (1-\omega) \, , 
\end{align}
where 
\begin{align*}
\omega:=\frac{\gamma_{i}}{\gamma_i-\gamma_{i-1}}
\end{align*}
and the superscript $n$ is momentarily omitted. To see this, consider first the case  $\sgn (\gamma_i)=\sgn(\gamma_{i-1})$. The sign of $u$ is then the same across $T_i$, hence according to (\ref{eq:def_ki}) we have $\overline{k_i(u^n)}^i= k_i(\gamma_{i}^n)=k_i(\gamma_{i-1}^n)$ and the claim follows for all $\omega$. For the case $\sgn(\gamma_i)\neq \sgn(\gamma_{i-1})$ a simple geometric consideration shows that the proportion $v$ of the interval $[x_{i-1}, x_i]$ where $\sgn(u^n)=\sgn(\gamma_i)$ equals the ratio of $\gamma_{i}$ and $\gamma_i-\gamma_{i-1}$ (see Figure \ref{fig:ill_w}): 
\begin{align*}
\frac{v}{h_i}=\frac{\gamma_{i}}{\gamma_i-\gamma_{i-1}} \, .
\end{align*}

\begin{figure}[ht]
    \centering
    \includegraphics[width=75mm]{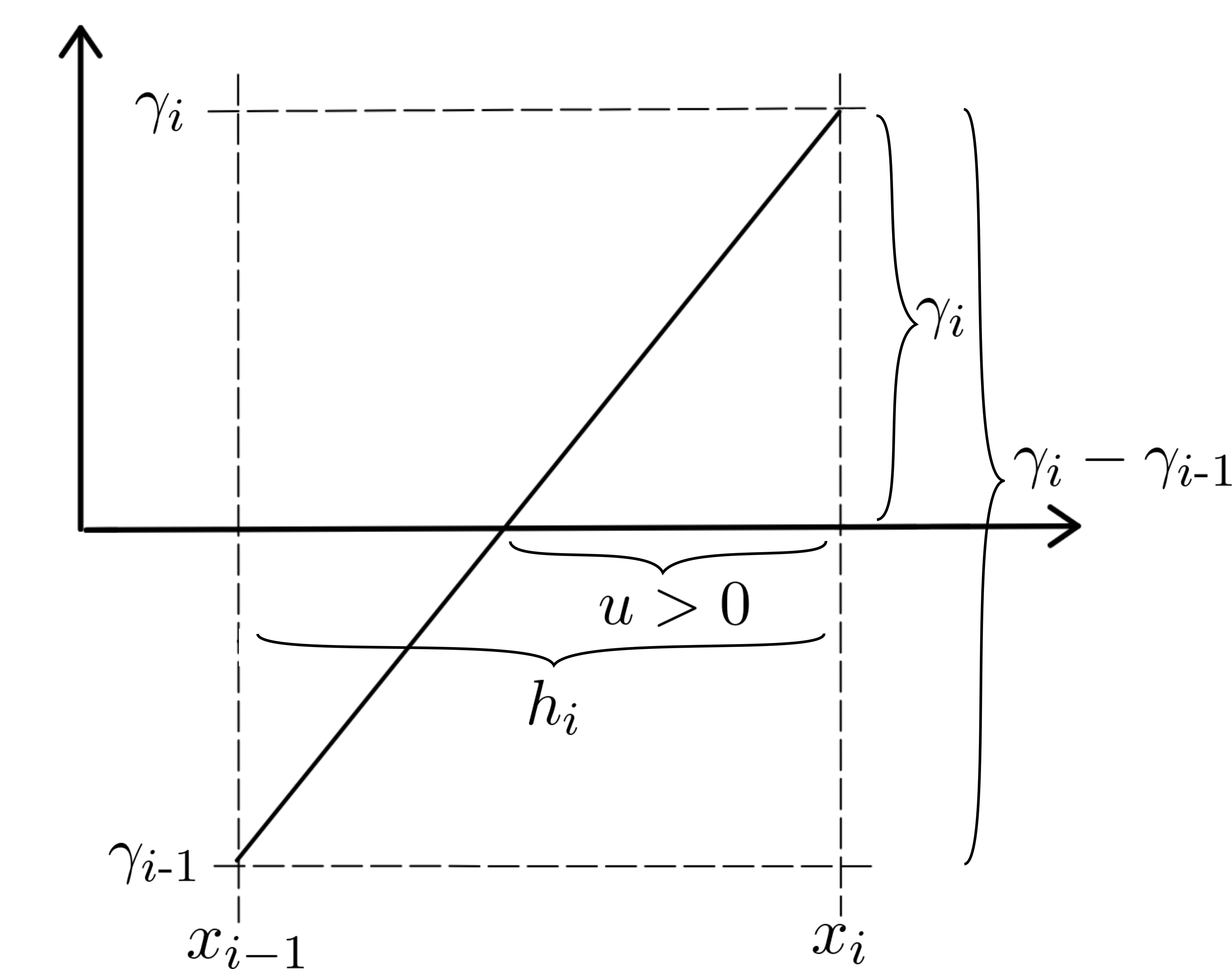}
    \caption{Illustration of the function $u$ over $T_i$, when a sign change occurs.}
    \label{fig:ill_w}
\end{figure}

Hence 
\begin{align*}
\overline{k_i(u^n)}^i \cdot h_i =  k_i(\gamma_{i}^n) \cdot v + k_i(\gamma_{i-1}^n)\cdot (h_i-v)
\end{align*}
and a division by $h_i$ shows (\ref{eq:sol_to_discr_weight}).
Plugging (\ref{eq:sol_to_discr_weight}) into (\ref{eq:def_algheatflux}) we obtain 
\begin{align*}
    &\frac{\gamma_{i}^n-\gamma_{i-1}^n }{h_i} \left[ k_i(\gamma_{i}^n)\frac{\gamma_{i}^n}{\gamma_i^n-\gamma_{i-1}^n} + k_i(\gamma_{i-1}^n)\left( 1-\frac{\gamma_{i}^n}{\gamma_i^n-\gamma_{i-1}^n}) \right) \right]\\
    & = \frac{ k_i(\gamma_{i}^n)\gamma_{i}^n}{h_i} +\frac{ k_i(\gamma_{i-1}^n)}{h_i}(\gamma_{i}^n-\gamma_{i-1}^n ) -\frac{k_i(\gamma_{i-1}^n)\gamma_{i}^n}{h_i}\\
    & =
    \frac{k_i(\gamma_{i}^n)\gamma_{i}^n-k_i(\gamma_{i-1}^n)\gamma_{i-1}^n }{h_i} \, = Q_i(\gamma^n).
\end{align*}

\section{Pseudocode of the used Katzenelson algorithm}
\label{App2}
The $\oslash$ sign denotes the element-wise division (Hadamard division). The $\text{rand}()$ notation stands for a random uniform distributed real number in the interval $[-1,1]$.
\begin{algorithm}[ht]
\caption{Katzenelson Algorithm.}
\label{alg:katzenelson}
\begin{algorithmic}[1]
\Require Initial guess $x$, function $\Phi(\cdot)$, Jacobian $D\Phi(\cdot)$, volumetric latent heat vector L, absolute tolerance $\text{t}_a$, relative tolerance $\text{t}_r$, characteristic enthalpy scale $s_x$.
  \State $J \gets D\Phi(x )$
  \State $r \gets \Phi(x)$
  \State $r_{ini} \gets \|r\| $
  \While{$\|r\| \;>\; r_{ini}\, \text{t}_r + \text{t}_a$}
    \State $\text{v} \gets - J^{-1}\,\Phi(x)$
    
    \State $\Lambda \gets \,(-x \oslash v, (\,
    L-x\,) \oslash v )$
    \State $\Lambda^s \gets \text{sort}(\{\Lambda_i\, | i \in \{1, \dots, 2\kappa\} \wedge \Lambda_i > \text{t}_r \wedge ( i>\kappa \Rightarrow L_{i-\kappa} >  s_x \,\text{t}_r)\} )$
    \State $\lambda_1 \gets \min(\Lambda^s_1, 1)$
    \State $\lambda_2 \gets \min(\Lambda^s_2, 1)$
    \If {$ |\lambda_1 - \lambda_2| > \text{t}_r \, \vee \lambda_1 > 1-\text{t}_r$}
          \State $J = D\Phi(x + \frac{\Lambda^{f}_1 + \Lambda^{f}_2}{2} v)$
          \State $\lambda \gets \Lambda^{f}_1$
    \Else
          \State $x \gets x + \frac{\Lambda_1^f}{2} v + \text{rand}() \, s_x \, 10^{-8}$
          \State $\lambda \gets 0$
    \EndIf
    \State $x \gets x + \lambda \,v$
    \State $r \gets \Phi(x)$
  \EndWhile
\end{algorithmic}
\end{algorithm}

For soil heat transfer the parameters $\text{t}_r = 10^{-12}$, $\text{t}_a = 10^{-6}$, $s_x = 10^{6}$ have turned out to work well and where used for most of the performed simulations.

\bibliographystyle{etna}
\bibliography{main}

\end{document}